\documentclass[12pt]{amsart}

\usepackage[lmargin=1in,rmargin=1in,tmargin=1in,bmargin=1in]{geometry}
\RequirePackage{amsmath} 
\RequirePackage{amssymb}
\usepackage{amscd,latexsym,amsthm,amsfonts,amssymb,amsmath,amsxtra}
\usepackage[colorlinks=true,urlcolor=blue,citecolor=blue]{hyperref}
\usepackage{color}
\usepackage[all]{xy}
\usepackage[OT2,T1]{fontenc}
\usepackage{bm}
\usepackage{mathtools}
\usepackage{mathrsfs}
\usepackage{xcolor}
\usepackage{comment}
\usepackage{marginnote}
\usepackage{enumitem}
\usepackage{thmtools, thm-restate}

\newcommand{\cO}{\mathcal{O}}

\newcommand{\bA}{\mathbb{A}}
\newcommand{\bR}{\mathbb{R}}
\newcommand{\Ad}{\text{Ad}}

\newcommand{\bQ}{\mathbb{Q}}
\newcommand{\Sym}{\text{Sym}}

\newcommand{\GL}{\text{GL}}

\newcommand{\cA}{\mathcal{A}}
\newcommand{\ov}{\overline}

\newtheorem{lemma}{Lemma}[section]
\newtheorem{remark}{Remark}[section]

\newtheorem{theorem}{Theorem}[section]

\newtheorem{thmx}{Theorem}

\title{Landau-Siegel Zeros of Triple Product L-functions}
\author{Shifan Zhao}

\begin{document}

\begin{abstract}
    Let $F$ be a number field. Let $\pi_1,\pi_2$ be cuspidal automorphic representations of $\GL_2(\bA_F)$, and let $\pi$ be a cuspidal automorphic representation of either $\GL_2(\bA_F)$ or $\GL_3(\bA_F)$. When $(\pi_1,\pi_2,\pi)$ is of general type, we show that the triple product $L$-function $L(s,\pi_1 \times \pi_2 \times \pi)$ on either $\GL(2) \times \GL(2) \times \GL(2)$ or $\GL(2) \times \GL(2) \times \GL(3)$ has a standard zero-free region with no exceptional Landau-Siegel zero. Moreover, when $(\pi_1,\pi_2,\pi)$ is not of general type, we give precise conditions when $L(s,\pi_1 \times \pi_2 \times \pi)$ could possibly have exceptional Landau-Siegel zeros. 
\end{abstract}

\maketitle

\tableofcontents

\section{Introduction}

\subsection{Background}
Let $F$ be a number field. Let $\pi$ be a cuspidal automorphic representation of $\GL_n(\bA_F)$, $n \geq 1$. Denote by $L(s,\pi)$ the standard $L$-function associated to $\pi$. Let $c>0$ be a positive number. We define a \textit{Landau-Siegel zero} of $L(s,\pi)$, relative to $c$, to be a real zero $\beta$ of $L(s,\pi)$ satisfying
\begin{equation*}
    1-\frac{c}{\log C(\pi)} < \beta < 1.
\end{equation*}
Here $C(\pi)$ denotes the analytic conductor of $L(s,\pi)$.

Studying Landau-Siegel zeros of automorphic $L$-functions is a central topic in analytic number theory. For example, when $F = \bQ$ and $\pi = \chi$ is a real primitive Dirichlet character, the possibility of existence of a Landau-Siegel zero of the Dirichlet $L$-function $L(s,\chi)$ is a long-standing open problem. Knowing the nonexistence or the existence of such a zero, on both sides, has deep consequences in other important arithmetic problems such as the distribution of prime numbers in arithmetic progressions \cite{Davenportbook}, the class number problem \cite{Davenportbook}, the twin prime conjecture \cite{Heath-Brown1983}, and the non-vanishing of $L$-functions on $\GL(2)$ \cite{IwaniecSarnak2000}, to name a few.

Although the problem of Landau-Siegel zeros of Dirichlet $L$-functions is extremely hard, eliminating Landau-Siegel zeros of higher degree automorphic $L$-functions is surprisingly more approachable. In the appendix of the ground-breaking paper \cite{HoffsteinLockhart1994}, Goldfeld, Hoffstein and Lieman proved that if $\pi$ is a non-dihedral cuspidal representation of $\GL_2(\bA_\bQ)$, then $L(s,\Ad(\pi))$ has no Landau-Siegel zero. Here $\Ad(\pi)$ denotes the adjoint square lift of $\pi$, which is a self-dual cuspidal representation of $\GL_3(\bA_\bQ)$ studied in \cite{GelbartJacquet1978}. Inspired by their work, Landau-Siegel zeros of standard $L$-functions $L(s,\pi)$ were extensively studied later in several papers including \cite{HoffsteinRamakrishnan1995}\cite{Banks1997}\cite{Thorner2021}\cite{Luo2023}. We will review some of these results in more details in Section \ref{Landau-Siegel zero known cases}.

Let $\pi_1$ and $\pi_2$ be cuspidal automorphic representations of $\GL_n(\bA_F)$ and $\GL_m(\bA_F)$, respectively. We similarly define a \textit{Landau-Siegel zero} of the Rankin-Selberg $L$-function $L(s,\pi_1 \times \pi_2)$, relative to a positive constant $c>0$, to be a real zero $\beta$ of $L(s,\pi_1 \times \pi_2)$ satisfying 
\begin{equation*}
    1-\frac{c}{\log (C(\pi_1)C(\pi_2))} < \beta < 1.
\end{equation*}

Landau-Siegel zeros of Rankin-Selberg $L$-functions were first studied by Ramakrishnan and Wang \cite{RamakrishnanWang2003}. They proved that for cuspidal representations $\pi_1$ and $\pi_2$ of $\GL_2(\bA_F)$, the Rankin-Selberg $L$-function $L(s,\pi_1 \times \pi_2)$ can only possibly have a Landau-Siegel zero in a few special cases. In particular, when the pair $(\pi_1,\pi_2)$ is of \textit{general type}, i.e., when $\pi_1$ and $\pi_2$ are non-dihedral and are not twist-equivalent, $L(s,\pi_1 \times \pi_2)$ has no Landau-Siegel zero. Landau-Siegel zeros of other types of Rankin-Selberg $L$-functions were studied in \cite{Luo2023}\cite{Thorner2025symmetricsquare}\cite{Thorner2025}. We will also review some of these results in more details in Section \ref{Landau-Siegel zero known cases}.

\subsection{Main Results}
In this paper we study Landau-Siegel zeros of \textit{triple product} $L$-functions on $\GL(2) \times \GL(2) \times \GL(2)$ and on $\GL(2) \times \GL(2) \times \GL(3)$.

Let $\pi_1,\pi_2$ and $\pi_3$ be cuspidal automorphic representations of $\GL_2(\bA_F)$. The triple product $L$-function $L(s,\pi_1 \times \pi_2 \times \pi_3)$ was first studied in \cite{Garrett1987} when $F = \bQ$ and $\pi_1,\pi_2,\pi_3$ associate to holomorphic cusp forms. The construction was then generalized to any number field $F$ in \cite{P-SRallis1987}. Moreover, in view of the work \cite{Ramakrishnan2000} on the modularity of Rankin-Selberg product on $\GL(2) \times \GL(2)$, the triple product $L$-function $L(s,\pi_1 \times \pi_2 \times \pi_3)$ can also be viewed as a Rankin-Selberg $L$-function on $\GL(4) \times \GL(2)$:
\begin{equation*}
    L(s,\pi_1 \times \pi_2 \times \pi_3) = L(s,(\pi_1 \boxtimes \pi_2) \times \pi_3),
\end{equation*}
whose analytic properties are well-known by \cite{JacquetP-SShalika1983}.

We say the triple $(\pi_1,\pi_2,\pi_3)$ is of 
\begin{itemize}
    \item \textit{general type}, if $\pi_1,\pi_2,\pi_3$ are all non-dihedral, and $\pi_i$ is not twist-equivalent to $\pi_j$ for all $i \neq j$.
    \item \textit{twist-equivalent type}, if $\pi_1,\pi_2,\pi_3$ are all non-dihedral, but $\pi_i$ is twist-equivalent to $\pi_j$ for some $i \neq j$.
    \item \textit{dihedral type}, if at least one of $\pi_1,\pi_2,\pi_3$ is dihedral. 
\end{itemize}

With these terminologies we are now ready to state our first main result. 
\begin{thmx}\label{Theorem A}
    There exists an absolute effective constant $c>0$, such that the triple product $L$-function $L(s,\pi_1 \times \pi_2 \times \pi_3)$ has no Landau-Siegel zero relative to $c$, except possibly in the following cases:
    \begin{enumerate}
        \item $(\pi_1,\pi_2,\pi_3)$ is of dihedral type.
        \item $(\pi_1,\pi_2,\pi_3)$ is of twist-equivalent type, and moreover $\pi_1,\pi_2,\pi_3$ are all twist-equivalent to each other.
    \end{enumerate}
\end{thmx}

\begin{remark}
    Theorem \ref{Theorem A} is a simplified version of the full statement Theorem \ref{Theorem A, refined version}. See Theorem \ref{Theorem A, refined version} for a precise description on the conditions when $L(s,\pi_1 \times \pi_2 \times \pi_3)$ could possibly have a Landau-Siegel zero. In particular, $L(s,\pi_1 \times \pi_2 \times \pi_3)$ has no Landau-Siegel zero if $(\pi_1,\pi_2,\pi_3)$ is of general type.
\end{remark}

\begin{remark}
    To illustrate why there might be a Landau-Siegel zero when $(\pi_1,\pi_2,\pi_3)$ is of dihedral type, let us assume $\pi_1$ is dihedral. Then $\pi_1 = I_K^F(\chi_1)$ is induced from a character $\chi_1$ on some quadratic extension $K/F$. By the adjointness property of Rankin-Selberg $L$-functions we have 
    \begin{equation*}
        L(s,\pi_1 \times \pi_2 \times \pi_3) = L(s,B_F^K(\pi_2) \times B_F^K(\pi_3) \otimes \chi_1),
    \end{equation*}
    where $B_F^K(\pi_2)$ and $B_F^K(\pi_3)$ are base changes from $F$ to $K$ of $\pi_2$ and $\pi_3$, respectively. Thus the triple product $L$-function reduces to a Rankin-Selberg $L$-function on $\GL(2) \times \GL(2)$ over $K$, a case that has been studied in \cite{RamakrishnanWang2003}. According to \cite[Theorem A]{RamakrishnanWang2003}, there are two reduced cases where there might be a Landau-Siegel zero.

    However, our argument in Section \ref{Section 3.3} gives more information on the possible Landau-Siegel zeros of $L(s,\pi_1 \times \pi_2 \times \pi_3)$ than those that can be obtained by directly applying \cite[Theorem A]{RamakrishnanWang2003}. This is due to the special natures of triple product $L$-functions.
\end{remark}

\begin{remark}
    To see why there might be a Landau-Siegel zero when $\pi_1,\pi_2,\pi_3$ are all twist-equivalent, let us simply assume that $\pi_1 = \pi_2 = \pi_3 = \pi$ is a non-dihedral cuspidal representation of $\GL_2(\bA_F)$ with central character $\omega_\pi$. In this case the triple product $L$-function decomposes as 
    \begin{equation*}
        L(s,\pi \times \pi \times \pi) = L(s,\Sym^3(\pi)) \cdot L(s,\pi \otimes \omega_\pi)^2.
    \end{equation*}
    However, eliminating Landau-Siegel zeros of the symmetric cube $L$-function $L(s,\Sym^3(\pi))$ is a long-standing open problem, which is only known under the (very strong) hypothesis that certain twisted symmetric fifth power $L$-function is analytic in some interval $(t,1)$, $t<1$.
\end{remark}

Let us proceed to describe our second main result. Let $\pi_1,\pi_2$ be cuspidal automorphic representations of $\GL_2(\bA_F)$. Let $\pi$ be a cuspidal automorphic representation of $\GL_3(\bA_F)$. Then the triple product $L$-function $L(s,\pi_1 \times \pi_2 \times \pi)$ on $\GL(2) \times \GL(2) \times \GL(3)$ can likewise be viewed as a Rankin-Selberg $L$-function on $\GL(4) \times \GL(3)$:
\begin{equation*}
    L(s,\pi_1 \times \pi_2 \times \pi) = L(s,(\pi_1 \boxtimes \pi_2) \times \pi).
\end{equation*}

We say the triple $(\pi_1,\pi_2,\pi)$ is of 
\begin{itemize}
    \item \textit{general type}, if $\pi_1,\pi_2$ are non-dihedral, not twist-equivalent, and $\pi$ is not twist-equivalent to either $\Ad(\pi_1)$ or $\Ad(\pi_2)$.
    \item \textit{twist-equivalent type}, if $\pi_1,\pi_2$ are non-dihedral, but either (1) $\pi_1$ is twist-equivalent to $\pi_2$, or (2) $\pi$ is twist-equivalent to $\Ad(\pi_1)$ or $\Ad(\pi_2)$.
    \item \textit{dihedral type}, if at least one of $\pi_1,\pi_2$ is dihedral.
\end{itemize}

Our second main result is as follows.

\begin{thmx}\label{Theorem B}
    There exists an absolute effective constant $c>0$ such that the triple product $L$-function $L(s,\pi_1 \times \pi_2 \times \pi)$ has no Landau-Siegel zero relative to $c$, except possibly in the case where $\pi$ is twist-equivalent to $\Ad(\pi_1)$ or $\Ad(\pi_2)$.
\end{thmx}

\begin{remark}
    See Theorem \ref{Theorem B, refined version} for a more refined version of Theorem \ref{Theorem B}, where we give a precise description on the conditions when $L(s,\pi_1 \times \pi_2 \times \pi)$ could possibly have a Landau-Siegel zero. In particular, $L(s,\pi_1 \times \pi_2 \times \pi)$ has no Landau-Siegel zero if $(\pi_1,\pi_2,\pi)$ is of general type or dihedral type.
\end{remark}

\begin{remark}
    It is interesting to see that $L(s,\pi_1 \times \pi_2 \times \pi)$ has no Landau-Siegel zero if $(\pi_1,\pi_2,\pi)$ is of dihedral type. To see why this is not obvious, let us assume $\pi_1 = I_K^F(\chi_1)$ is dihedral. Then we have 
    \begin{equation*}
        L(s,\pi_1 \times \pi_2 \times \pi) = L(s,B_F^K(\pi_2) \times B_F^K(\pi) \otimes \chi_1).
    \end{equation*}
    Landau-Siegel zeros of the $\GL(2) \times \GL(3)$ Rankin-Selberg $L$-function on the right side have been studied in \cite{Luo2023}. According to \cite[Theorem 2]{Luo2023}, there might be a Landau-Siegel zero if $B_F^K(\pi_2)$ is cuspidal non-dihedral, and $B_F^K(\pi)$ is twist-equivalent to $\Ad(B_F^K(\pi_2))$. However, using a different argument we are able to show that there is no Landau-Siegel zero at all in this dihedral case. This argument, again, relies on special natures of triple product $L$-functions.
\end{remark}

\begin{remark}
    To see why there might be a Landau-Siegel zero when $\pi$ is twist-equivalent to $\Ad(\pi_1)$ or $\Ad(\pi_2)$, let us assume that $\pi = \Ad(\pi_1)$, and $\pi_1,\pi_2$ are not twist-equivalent. In this case we have 
    \begin{equation*}
        L(s,\pi_1 \times \pi_2 \times \pi) = L(s,\Sym^3(\pi_1) \times \pi_2 \otimes \omega_1^{-1}) \cdot L(s,\pi_1 \times \pi_2),
    \end{equation*}
    where $\omega_1$ is the central character of $\pi_1$. Eliminating Landau-Siegel zeros of the Rankin-Selberg $L$-function $L(s,\Sym^3(\pi_1) \times \pi_2 \otimes \omega_1^{-1})$ is an open problem.
\end{remark}

In \cite{Thorner2025}\cite{Thorner2025symmetricsquare}, Thorner addressed the importance of allowing an additional character twist to get a standard zero-free region. In our case, let $\pi_1,\pi_2$ be cuspidal automorphic representations on $\GL_2(\bA_F)$ and $\pi$ be a cuspidal automorphic representation on $\GL_2(\bA_F)$ or $\GL_3(\bA_F)$. Let $\eta$ be an arbitrary (idele class) character on $F$. Then in view of the relation
$$L(s,\pi_1 \times \pi_2 \times \pi \otimes \eta) = L(s,\pi_1 \times \pi_2 \times (\pi \otimes \eta)),$$
we can study the Landau-Siegel zeros of $L(s,\pi_1 \times \pi_2 \times \pi \otimes \eta)$ by applying Theorem \ref{Theorem A} and Theorem \ref{Theorem B} to the triple $(\pi_1,\pi_2,\pi \otimes \eta)$. In particular, if $(\pi_1,\pi_2,\pi)$ is of general type, then $(\pi_1,\pi_2,\pi \otimes \eta)$ is also of general type. By taking $\eta = |\cdot|^{it}$, $t \in \bR$, we immediately conclude the following:

\begin{thmx}
     Assume $(\pi_1,\pi_2,\pi)$ is of general type. Then there exists an effective absolute constant $c>0$, such that $L(s,\pi_1 \times \pi_2 \times \pi) \neq 0$ in the following region:
    $$\left\{s = \sigma +it: \sigma > 1-\frac{c}{\log (C(\pi_1)C(\pi_2)C(\pi)(|t|+3)^{[F:\bQ]})}\right\}.$$
\end{thmx}

\subsection{Proof Strategy}

As is well-known to experts, the technical heart behind any attempt to eliminate Landau-Siegel zeros of an $L$-function $L(s,\pi)$ (or $L(s,\pi_1 \times \pi_2)$) is to construct an auxiliary $L$-function $L(s)$ with at most $r$ Landau-Siegel zeros, in such a way that we can decompose $L(s)$ as 
\begin{equation*}
    L(s) = L(s,\pi)^a \cdot L(s,\ov{\pi})^b \cdot L_{Res}(s),
\end{equation*}
with $a+b > r$, and the residual part $L_{Res}(s)$ is analytic in some interval $(t,1)$, $t<1$. If we have constructed such an $L(s)$, then any Landau-Siegel zero of $L(s,\pi)$ would be a Landau-Siegel zero of order at least $a+b$ of $L(s)$, which cannot exist because $L(s)$ has at most $r$ Landau-Siegel zeros, and $r<a+b$. 

For most of the work in this field, including the present paper, the auxiliary $L$-function is constructed as a special type of Rankin-Selberg $L$-function, $L(s) = L(s,\Pi \times \ov{\Pi})$, where $\Pi$ is some carefully chosen isobaric automorphic representation. Thus if there is anything creative in this paper, it is the construction of $\Pi$, which, in the new context of triple product $L$-functions, exhibits some new features.  

Recently, in his preprint \cite{Thorner2025} towards the study of Landau-Siegel zeros of Rankin-Selberg $L$-functions of symmetric power liftings of Hilbert modular forms over totally real fields, Thorner (for the first time) constructed auxiliary $L$-functions $L(s)$ that are not in the shape of $L(s,\Pi \times \ov{\Pi})$. In \cite{Thorner2025symmetricsquare}, he also studied Landau-Siegel zeros of Rankin-Selberg $L$-functions $L(s,\Ad(\pi_1) \times \Ad(\pi_2) \otimes \chi)$, where $\pi_1,\pi_2$ are cuspidal representations of $\GL_2(\bA_F)$, and $\chi$ is a character on $F$. As we shall see later, our study of Landau-Siegel zeros of $\GL(2) \times \GL(2) \times \GL(3)$ triple product $L$-functions reduces to this type in some particular cases. So we would like to express our gratefulness to Prof. Jesse Thorner for his great work. 

This paper is organized as follows. In Section \ref{Section 2} we review some preliminaries, including some instances of Langlands functoriality and some known cases of the nonexistence of Landau-Siegel zeros that will be used throughout this paper. In Section \ref{Section 3} and Section \ref{Section 4} we prove Theorem \ref{Theorem A} and Theorem \ref{Theorem B}, respectively.

\section{Preliminaries}\label{Section 2}

\subsection{Langlands Functoriality}
In this section we review some known instances of Langlands functoriality that are crucial to our work. First let us introduce some notations. We use $\cA_0(n,F)$ (resp. $\cA(n,F)$) to denote the set of cuspidal (resp. isobaric) automorphic representations of $\GL_n(\bA_F)$ with unitary central characters $\omega_\pi$. For $\pi \in \cA(n,F)$ we denote by $\ov{\pi}$ the \textit{dual} representation of $\pi$. We say $\pi$ is \textit{self-dual}, if $\pi \cong \ov{\pi}$. We say $\pi$ is \textit{monomial}, if $\pi \cong \pi \otimes \eta$ for some non-trivial character $\eta$ on $F$. We say $\pi_1,\pi_2 \in \cA_0(n,F)$ are \textit{twist-equivalent}, if $\pi_1 \cong \pi_2 \otimes \mu$ for some character $\mu$ on $F$.

Let us start with base change and automorphic induction. Let $K/F$ be a cyclic extension of number fields of prime degree $\ell$. Let $n \geq 1$ be a positive integer. The base change map 
\begin{equation*}
    B_F^K:\cA(n,F) \to \cA(n,K), \hspace{3mm} \pi \mapsto B_F^K(\pi)
\end{equation*}
and the automorphic induction map 
\begin{equation*}
    I_K^F:\cA(n,K) \to \cA(n\ell,F), \hspace{3mm} \alpha \mapsto I_K^F(\alpha)
\end{equation*}
were constructed in \cite{ArthurClozel1989}. These maps satisfy the following properties: 

\begin{lemma}\label{base change and automporphic induction}
    Let $\theta = \theta_{K/F}$ be a generator of the Galois group Gal$(K/F)$. Let $\eta = \eta_{K/F}$ be the character on $F$ associated to $K/F$. Then 
    \begin{enumerate}
        \item Let $\pi \in \cA(n\ell,F)$. Then $\pi$ lies in the image of $I_K^F$ if and only if $\pi \cong \pi \otimes \eta$.
        \item Let $\alpha \in \cA(n,K)$. Then $\alpha$ lies in the image of $B_F^K$ if and only if $\alpha \cong \alpha \circ \theta$.
        \item For $\pi \in \cA(n,F)$ and $\alpha \in \cA(m,K)$, we have 
        \begin{equation*}
            L(s,\pi \times I_K^F(\alpha)) = L(s,B_F^K(\pi) \times \alpha).
        \end{equation*}
        \item Let $\pi \in \cA_0(n,F)$ be cuspidal. Then $B_F^K(\pi)$ is cuspidal if and only if $\pi$ is not isomorphic to $\pi \otimes \eta$. Moreover, we have 
        $$I_K^F(B_F^K(\pi)) \cong \boxplus_{j=0}^{\ell-1} \pi \otimes \eta^j.$$
        \item Let $\alpha \in \cA_0(n,K)$ be cuspidal. Then $I_K^F(\alpha)$ is cuspidal if and only if $\alpha$ is not isomorphic to $\alpha \circ \theta$. Moreover, we have 
        $$B_F^K(I_K^F(\alpha)) \cong \boxplus_{j=0}^{\ell-1} \alpha \circ \theta^j.$$
    \end{enumerate}
\end{lemma}

We then review functoriality results of symmetric power liftings. Let $\pi \in \cA_0(2,F)$. Denote by $\Sym^m(\pi)$ the symmetric $m^{th}$-power lifting of $\pi$. We set $\Ad(\pi) := \Sym^2(\pi) \otimes \omega_\pi^{-1}$, $A^3(\pi) := \Sym^3(\pi) \otimes \omega_\pi^{-1}$ and $A^4(\pi) := \Sym^4(\pi) \otimes \omega_\pi^{-2}$. We say $\pi$ is 
\begin{itemize}
    \item \textit{dihedral}, if $\pi$ is monomial, or equivalently, if $\pi = I_K^F(\chi)$ is an automorphic induction of a character $\chi$ on some quadratic extension $K/F$.
    \item \textit{tetrahedral}, if $\Sym^2(\pi)$ is cuspidal and monomial.
    \item \textit{octahedral}, if $\Sym^3(\pi)$ is cuspidal and monomial. 
    \item of \textit{solvable polyhedral type}, if $\pi$ is either dihedral or tetrahedral or octahedral.
\end{itemize}
Below is a summary of \cite[Theorem 9.3]{GelbartJacquet1978},\cite[Theorem B]{KimShahidi2002Cube},\cite[Theorem 2.2.2]{KimShahidi2002FourthPowerCuspidality},\cite[Theorem B]{Kim2003fourth},\cite[Theorem 3.3.7]{KimShahidi2002FourthPowerCuspidality} and \cite[Theorem A]{NewtonThorne2025}:

\begin{lemma}\label{symmetric power functorality}
    Let $\pi \in \cA_0(2,F)$. Then we have 
    \begin{enumerate}
        \item $\Sym^2(\pi) \in \cA(3,F)$. $\Sym^2(\pi)$ is not cuspidal if and only if $\pi$ is dihedral. Moreover, $\Ad(\pi)$ is self-dual.
        \item $\Sym^3(\pi) \in \cA(4,F)$. $\Sym^3(\pi)$ is not cuspidal if and only if $\pi$ is of dihedral or tetrahedral type. Moreover, when $\pi$ is of tetrahedral type, write $\Ad(\pi) \cong \Ad(\pi) \otimes \eta$ for some cubic character $\eta$ on $F$. In this case we have
        $$A^3(\pi) = (\pi \otimes \eta) \boxplus (\pi \otimes \eta^2).$$
        \item $\Sym^4(\pi) \in \cA(5,F)$. $\Sym^4(\pi)$ is not cuspidal if and only if $\pi$ is of solvable polyhedral type. Moreover,
        \begin{enumerate}
            \item When $\pi$ is of tetrahedral type, write $\Ad(\pi) \cong \Ad(\pi) \otimes \eta$ for some cubic character $\eta$ on $F$. In this case we have 
            $$A^4(\pi) = \Ad(\pi) \boxplus\eta \boxplus \eta^2.$$
            \item When $\pi$ is of octahedral type, write $\Sym^3(\pi) \cong \Sym^3(\pi) \otimes \eta$ for some quadratic character $\eta$ on $F$. Let $K/F$ be the quadratic extension associated with $\eta$. Then $\Ad(B_F^K(\pi)) \cong \Ad(B_F^K(\pi)) \otimes \chi$ for some cubic character $\chi$ on $K$. In this case we have 
            $$A^4(\pi) = I_K^F(\chi^{-1}) \boxplus (\Ad(\pi) \otimes \eta).$$
        \end{enumerate}
        \item Assume $F$ is totally real, and $\pi$ is non-dihedral regular algebraic. Then $\Sym^{m-1}(\pi) \in \cA_0(m,F)$ for all $m \geq 2$.
    \end{enumerate}
\end{lemma}

Finally we review functoriality results of Rankin-Selberg products on $\GL(2) \times \GL(2)$ and on $\GL(2) \times \GL(3)$. Below is a summary of \cite[Theorem M]{Ramakrishnan2000}, \cite[Theorem A]{KimShahidi2002Cube} and \cite[Theorem 3.1]{RamakrishnanWang2004}:

\begin{lemma}\label{Rankin-Selberg product functorality}
    Let $\pi_1,\pi_2 \in \cA_0(2,F)$. Let $\pi \in \cA_0(3,F)$. Then we have 
    \begin{enumerate}
        \item $\pi_1 \boxtimes \pi_2 \in \cA(4,F)$. Moreover, $\pi_1 \boxtimes \pi_2$ is cuspidal if and only if one the following happens:
        \begin{enumerate}
            \item $\pi_1,\pi_2$ are both non-dihedral, and $\pi_1$ is not twist-equivalent to $\pi_2$.
            \item One of $\pi_1,\pi_2$, say $\pi_1$, is dihedral. Then $\pi_1 = I_K^F(\chi_1)$ for a character $\chi_1$ on some quadratic extension $K/F$. And the base change $B_F^K(\pi_2)$ is cuspidal and not isomorphic to $B_F^K(\pi_2) \otimes (\chi_1 \circ\theta_{K/F}) \chi_1^{-1}$.
        \end{enumerate}
        \item $\pi_1 \boxtimes \pi \in \cA(6,F)$. Moreover, $\pi_1 \boxtimes \pi$ is not cuspidal if and only if one of the following happens:
        \begin{enumerate}
            \item $\pi_1$ is non-dihedral, and $\pi$ is twist-equivalent to $\Ad(\pi_1)$.
            \item $\pi_1$ is dihedral, $L(s,\pi) = L(s,\chi)$ for a character $\chi$ on a cubic, non-normal extension $K/F$, and the base change $B_F^K(\pi)$ satisfies $L(s,B_F^K(\pi)) = L(s,\psi)$ for a character $\psi$ on a quadratic extension $L/K$.
        \end{enumerate}
    \end{enumerate}
\end{lemma}

At the end of this section we prepare the following lemma for future reference:

\begin{lemma}\label{multiplicity one}
        Let $\pi_1,\pi_2 \in \cA_0(2,F)$ be non-dihedral. Suppose that $\Ad(\pi_1)$ and $\Ad(\pi_2)$ are twist-equivalent. Then $\pi_1$ and $\pi_2$ are also twist-equivalent.
\end{lemma}

 \begin{proof}
        Suppose $\Ad(\pi_1) \cong \Ad(\pi_2) \otimes \eta$ for some character $\eta$ on $F$. Computing the central characters of both sides we have 
        $$1 = \omega_{\Ad(\pi_1)} = \omega_{\Ad(\pi_2) \otimes \eta} = \omega_{\Ad(\pi_2)} \eta^3 = \eta^3.$$
        If $\eta = 1$, then $\Ad(\pi_1) \cong \Ad(\pi_2)$, which would imply $\pi_1$ and $\pi_2$ are twist-equivalent by the multiplicity one theorem \cite[Theorem 4.1.2]{Ramakrishnan2000}. If $\eta \neq 1$, then for each finite place $v$ where $\pi_1$, $\pi_2$ and $\eta$ are unramified we have the following two sets of local parameters of $\Ad(\pi_1)$ and $\Ad(\pi_2) \otimes \eta$ at $v$
        $$\{\alpha_{1,v},1,\alpha_{1,v}^{-1}\} = \{\alpha_{2,v}\eta_v,\eta_v,\alpha_{2,v}^{-1}\eta_v\}$$
        are the same. Thus either $\eta_v = \alpha_{1,v}$ or $\eta_v = \alpha_{1,v}^{-1}$. By symmetry we may assume that $\eta_v = \alpha_{1,v}$. Thus we have 
        $$\{\alpha_{1,v}\eta_v,\eta_v,\alpha_{1,v}^{-1}\eta_v\} = \{\alpha_{1,v},1,\alpha_{1,v}^2\} = \{\alpha_{1,v},1,\alpha_{1,v}^{-1}\},$$
        which implies that 
        $$\Ad(\pi_{1,v}) \cong \Ad(\pi_{1,v}) \otimes \eta_v \cong \Ad(\pi_{2,v}) \otimes \eta_v.$$
        Thus we have 
        $$\Ad(\pi_{1,v}) \cong \Ad(\pi_{2,v})$$
        for all but finitely many $v$. By \cite[Theorem 4.1.2]{Ramakrishnan2000} this would also imply that $\pi_1$ and $\pi_2$ are twist equivalent.
\end{proof}

\subsection{Known Cases of Landau-Siegel Zeros}\label{Landau-Siegel zero known cases}
In this section we review some known cases where Landau-Siegel zeros can be eliminated. Our study of Landau-Siegel zeros of triple product $L$-functions reduces to them in some special cases. 

Let us first state results on standard $L$-functions. Below is a summary of \cite[Corollary 3.2]{HoffsteinRamakrishnan1995}, \cite[Theorem 3]{Luo2023}, \cite[Theorem C]{HoffsteinRamakrishnan1995} and \cite[Theorem 1]{Banks1997}:

\begin{lemma}\label{standard L-function Landau-Siegel zero}
    Let $\pi \in \cA_0(n,F)$. Then $L(s,\pi)$ has no Landau-Siegel zero in the following cases:
    \begin{enumerate}
        \item $\pi$ is non-self-dual.
        \item $n \geq 2$ and $\pi$ is monomial. 
        \item $n=2,3$.
    \end{enumerate}
\end{lemma}

We then state results on Rankin-Selberg $L$-functions. Below is a summary of \cite[Theorem A]{RamakrishnanWang2003},\cite[Theorem 1]{Luo2023} and \cite[Theorem 2]{Luo2023}:

\begin{lemma}\label{Rankin Selberg Landau-Siegel zero}
    Let $\pi_1,\pi_2 \in \cA_0(2,F)$. Let $\pi \in \cA_0(3,F)$. Then we have 
    \begin{enumerate}
        \item $L(s,\pi_1 \times \pi_2)$ has no Landau-Siegel zero except possibly in the following cases:
        \begin{enumerate}
            \item $\pi_1,\pi_2$ are non-dihedral, and $\pi_1$ is twist-equivalent to $\pi_2$.
            \item $\pi_1$ and $\pi_2$ are both dihedral, and $\pi_1 = I_K^F(\chi_1), \pi_2 = I_K^F(\chi_2)$ for characters $\chi_1,\chi_2$ on a single quadratic extension $K/F$.
        \end{enumerate}
        \item $L(s,\pi_1 \times \pi)$ has no Landau-Siegel zero except possibly in the case when $\pi_1$ is non-dihedral, and $\pi$ is twist-equivalent to $\Ad(\pi_1)$.
    \end{enumerate}
\end{lemma}

All the results stated above rely on the following lemma, which is a combination of \cite[Lemma a]{HoffsteinRamakrishnan1995} and \cite[Appendix, Lemma]{HoffsteinLockhart1994}.

\begin{lemma}\label{Landau-Siegel zero lemma}
    Let $\Pi \in \cA(n,F)$. Write $\Pi = \boxplus_{j=1}^J e_j\pi_j$ as an isobaric sum of cuspidal representations, where $\pi_i$ is not isomorphic to $\pi_j$ when $i \neq j$, and $e_j \geq 1$. Then there exists an effective constant $c = c(n,e_j)$ depending only on $n$ and the $e_j$'s, such that $L(s,\Pi \times \ov{\Pi})$ has at most $e = \sum_{j=1}^J e_j^2$ many Landau-Siegel zeros relative to $c$. Moreover, if $L(s,\Pi \times \ov{\Pi})$ decomposes as 
    $$L(s,\Pi \times \ov{\Pi}) = L_1(s)^a \cdot L_2(s)^b \cdot L_{Res}(s),$$
    where $a+b > e$, $L_1(s)$ shares the same real zeros with $L_2(s)$, and $L_{Res}(s)$ is analytic in some interval $(t,1)$, $t<1$, then $L_1(s)$ has no Landau-Siegel zero.
\end{lemma}

At the end of this section let us prove the following simple lemma that will be used later:

\begin{lemma}\label{new Rankin-Selberg Landau-Siegel zero}
    Let $\pi_1 \in \cA_0(n,F)$ and $\pi_2 \in \cA_0(m,F)$. Then $L(s,\pi_1 \times \pi_2)$ has no Landau-Siegel zero if one of the following happens:
    \begin{enumerate}
        \item $\pi_1$ is self-dual, and $\pi_2$ is non-self-dual. 
        \item $\pi_1$ is monomial with $\pi_1 \cong \pi_1 \otimes \eta$ for a non-trivial character $\eta$ on $F$, and $\pi_2$ is not isomorphic to $\pi_2 \otimes \eta$.
    \end{enumerate}
\end{lemma}

\begin{proof}
    \begin{enumerate}
        \item Since $\pi_1$ is self-dual and $\pi_2$ is not, no two of the three representations $\pi_1,\pi_2, \ov{\pi_2}$ are isomorphic. Set $\Pi = \pi_1 \boxplus \pi_2 \boxplus \ov{\pi_2}$. Note that $L(s,\Pi \times \ov{\Pi})$ decomposes as 
        \begin{align*}
            L(s,\Pi \times \ov{\Pi}) &= L(s,\pi_1 \times \pi_2)^2 \cdot L(s,\pi_1 \times \ov{\pi_2})^2 \\
            &\cdot L(s,\pi_1 \times \pi_1) \cdot L(s,\pi_2 \times \pi_2) \cdot L(s,\ov{\pi_2} \times \ov{\pi_2}) \cdot L(s,\pi_2 \times \ov{\pi_2})^2.
        \end{align*}
        Since $\pi_1$ is self-dual, $L(s,\pi_1 \times \pi_2)$ and $L(s,\pi_1 \times \ov{\pi_2})$ have the same real zeros. Thus $L(s,\pi_1 \times \pi_2)$ has no Landau-Siegel zero by Lemma \ref{Landau-Siegel zero lemma}.
        \item Since $\pi_1 \cong \pi_1 \otimes \eta$ and $\pi_2$ is not isomorphic to $\pi_2 \otimes \eta$, no two of $\ov{\pi_1},\pi_2,\pi_2 \otimes \eta$ are isomorphic. Set $\Pi = \ov{\pi_1} \boxplus \pi_2 \boxplus (\pi_2 \otimes \eta)$. In view of the relation $\pi_1 \cong \pi_1 \otimes \eta$, $L(s,\Pi \times \ov{\Pi})$ decomposes as 
        \begin{align*}
            L(s,\Pi \times \ov{\Pi}) &= L(s,\pi_1 \times \pi_2)^2 \cdot L(s,\ov{\pi_1} \times \ov{\pi_2})^2 \\
            &\cdot L(s,\pi_1 \times \ov{\pi_1}) \cdot L(s,\pi_2 \times \ov{\pi_2})^2 \cdot L(\pi_2 \times \ov{\pi_2} \otimes \eta) \cdot L(s,\pi_2 \times \ov{\pi_2} \otimes \eta^{-1}).
        \end{align*}
        Thus $L(s,\pi_1 \times \pi_2)$ has no Landau-Siegel zero by Lemma \ref{Landau-Siegel zero lemma}.
    \end{enumerate}
\end{proof}

\section{Proof of Theorem A}\label{Section 3}

The goal of this section is to prove the following refined version of Theorem \ref{Theorem A}:

\begin{theorem}\label{Theorem A, refined version}
    Let $\pi_1,\pi_2,\pi_3 \in \cA_0(2,F)$ with central characters $\omega_1,\omega_2,\omega_3$ respectively. Then $L(s,\pi_1 \times \pi_2 \times \pi_3)$ has no Landau-Siegel zero except possibly in the following cases:
    \begin{enumerate}
            \item There exists a quadratic extension $K/F$ such that $\pi_i = I_K^F(\chi_i)$ for some character $\chi_i$ on $K$, for each $i=1,2,3$. In this case there are at most two Landau-Siegel zeros coming from    $$L(s,\chi_1\chi_2\chi_3)L(s,\chi_1\chi_2^\theta\chi_3)L(s,\chi_1\chi_2\chi_3^\theta)L(s,\chi_1\chi_2^\theta\chi_3^\theta).$$
            Here $\theta = \theta_{K/F}$ is the non-trivial element in Gal$(K/F)$, and wet set $\chi_i^\theta := \chi_i \circ \theta$. Moreover, there is at most one Landau-Siegel zero if $(\chi_2\chi_3)^{\theta} \neq \chi_2\chi_3$.
            \item There exists a quadratic extension $K/F$ such that one of $\pi_1,\pi_2,\pi_3$, say $\pi_1$, is induced from some character $\chi_1$ on $K$, and $B_F^K(\pi_2) \otimes \chi_1$ and $B_F^K(\pi_3)$ are both cuspidal and dihedral, induced form characters $\xi_2,\xi_3$ on $L$ respectively, where $L/K$ is the quadratic extension associated with the quadratic character $\chi_1^\theta\chi_1^{-1}$ on $K$. In this case there is at most one Landau-Siegel zero coming from 
            $$L(s,\xi_3\xi_2)L(s,\xi_3\xi_2^{\theta_{L/K}}).$$
            \item $\pi_1,\pi_2,\pi_3$ are not of solvable polyhedral type, and are all twist-equivalent. Moreover, suppose $\pi_2 \cong \pi_1 \otimes \eta$ and $\pi_3 \cong \pi_1 \otimes \mu$ for characters $\eta$ and $\mu$ on $F$. Then there is no Landau-Siegel zero if the twisted symmetric fifth power $L$-function
            $$L(s,\pi_1;\Sym^5 \otimes \omega_1^{-1}\eta\mu)$$
            has no pole in some interval $(t,1)$, $t<1$.
        \end{enumerate}
\end{theorem}

\begin{remark}
    The condition in Theorem \ref{Theorem A, refined version} (3) is satisfied if $F$ is totally real, and $\pi_1$ is regular algebraic by Lemma \ref{symmetric power functorality} (4). Thus there is no Landau-Siegel zero in this case.
\end{remark}

\subsection{General Type}
Assume $(\pi_1,\pi_2,\pi_3)$ is of general type. Recall that this means $\pi_1,\pi_2,\pi_3$ are all non-dihedral, and $\pi_i$ is not twist-equivalent to $\pi_j$ for all $i \neq j$. We set
$$\Pi := (\pi_1 \boxtimes \pi_2) \boxplus \ov{\pi_3} \boxplus (\Ad(\pi_1) \boxtimes \ov{\pi_3}).$$
Now since $\pi_1,\pi_2$ are non-dihedral and are not twist-equivalent, $\pi_1 \boxtimes \pi_2$ is cuspidal by Lemma \ref{Rankin-Selberg product functorality} (1). We then claim that $\Ad(\pi_1) \boxtimes \ov{\pi_3}$ is also cuspidal. Suppose not, then by Lemma \ref{Rankin-Selberg product functorality} (2), $\Ad(\pi_1)$ is twist-equivalent to $\Ad(\ov{\pi_3}) = \Ad(\pi_3)$. By Lemma \ref{multiplicity one} this would imply that $\pi_1,\pi_2$ are twist-equivalent, a contradiction!

Now the three representations $\pi_1 \boxtimes \pi_2, \ov{\pi_3}$ and $\Ad(\pi_1) \boxtimes \ov{\pi_3}$ are pairwise non-isomorphic, since they have different degrees. We have
\begin{align*}
    L(s,\Pi \times \ov{\Pi}) &= L(s,\pi_1 \times \pi_2 \times \pi_3) \cdot L(s,\ov{\pi_1} \times \ov{\pi_2} \times \ov{\pi_3}) \\
    &\cdot L(s,(\pi_1 \boxtimes \pi_2) \times (\Ad(\pi_1) \boxtimes \pi_3)) \cdot L(s,(\ov{\pi_1} \boxtimes \ov{\pi_2}) \times (\Ad(\pi_1) \boxtimes \ov{\pi_3})) \\
    &\cdot L(s,\ov{\pi_3} \times (\Ad(\pi_1) \boxtimes \pi_3))^2 \\
    &\cdot L(s, (\pi_1 \boxtimes \pi_2) \times (\ov{\pi_1} \boxtimes \ov{\pi_2})) \cdot L(s,\pi_3 \times \ov{\pi_3}) \cdot L(s,(\Ad(\pi_1) \boxtimes \pi_3) \times (\Ad(\pi_1) \boxtimes \ov{\pi_3})).
\end{align*}
Note that we can further decompose $L(s,(\pi_1 \boxtimes \pi_2) \times (\Ad(\pi_1) \boxtimes \pi_3))$ as 
\begin{align*}
    L(s,(\pi_1 \boxtimes \pi_2) \times (\Ad(\pi_1) \boxtimes \pi_3)) &= L(s,(\pi_1 \boxtimes \Ad(\pi_1)) \times (\pi_2 \boxtimes \pi_3)) \\
    &= L(s,(\pi_1 \boxplus A^3(\pi_1)) \times (\pi_2 \boxtimes \pi_3)) \\
    &= L(s,\pi_1 \times \pi_2 \times \pi_3) \cdot L(s,A^3(\pi_1) \times (\pi_2 \boxtimes \pi_3)).
\end{align*}
Thus $L(s,\Pi \times \ov{\Pi})$ is equal to 
\begin{align*}
    L(s,\Pi \times \ov{\Pi}) &= L(s,\pi_1 \times \pi_2 \times \pi_3)^2 \cdot L(s,\ov{\pi_1} \times \ov{\pi_2} \times \ov{\pi_3})^2 \\
    &\cdot L(s,A^3(\pi_1) \times (\pi_2 \boxtimes \pi_3)) \cdot L(s,\ov{A^3(\pi_1)} \times (\ov{\pi_2} \boxtimes \ov{\pi_3})) \\
    &\cdot L(s,\ov{\pi_3} \times (\Ad(\pi_1) \boxtimes \pi_3))^2 \\
    &\cdot L(s, (\pi_1 \boxtimes \pi_2) \times (\ov{\pi_1} \boxtimes \ov{\pi_2})) \cdot L(s,\pi_3 \times \ov{\pi_3}) \cdot L(s,(\Ad(\pi_1) \boxtimes \pi_3) \times (\Ad(\pi_1) \boxtimes \ov{\pi_3})).
\end{align*}
By Lemma \ref{Landau-Siegel zero lemma}, $L(s,\pi_1 \times \pi_2 \times \pi_3)$ has no Landau-Siegel zero.

\subsection{Twist-equivalent Type}
Assume $(\pi_1,\pi_2,\pi_3)$ is of twist-equivalent type. Recall that this means $\pi_1,\pi_2,\pi_3$ are all non-dihedral, but $\pi_i$ is twist-equivalent to $\pi_j$ for some $i \neq j$. Without loss of generality let us assume that $\pi_1$ and $\pi_2$ are twist-equivalent. Write $\pi_2 \cong \pi_1 \otimes \eta$ for some character $\eta$ on $F$. Then $L(s,\pi_1 \times \pi_2 \times \pi_3)$ decomposes as 
\begin{align*}
    L(s,\pi_1 \times \pi_2 \times \pi_3) &= L(s,\pi_1 \times (\pi_1 \otimes \eta) \times \pi_3) \\
    &= L(s,(\pi_1 \boxtimes \pi_1) \times \pi_3 \otimes \eta) \\
    &= L(s,(\Ad(\pi_1) \otimes \omega_1 \boxplus \omega_1) \times \pi_3 \otimes \eta) \\
    &= L(s,\Ad(\pi_1) \times \pi_3 \otimes \omega_1\eta) \cdot L(s,\pi_3 \otimes \omega_1\eta).
\end{align*}
Now $L(s,\pi_3 \otimes \omega_1\eta)$ has no Landau-Siegel zero by Lemma \ref{standard L-function Landau-Siegel zero} (3). Moreover, by Lemma \ref{Rankin Selberg Landau-Siegel zero} (2), $L(s,\Ad(\pi_1) \times \pi_3 \otimes \omega_1\eta)$ can only possibly have Landau-Siegel zeros when $\Ad(\pi_1)$ is twist-equivalent to $\Ad(\pi_3)$, which would imply that $\pi_1$ is twist-equivalent to $\pi_3$ by Lemma \ref{multiplicity one}. As a conclusion, $L(s,\pi_1 \times \pi_2 \times \pi_3)$ can only possibly have Landau-Siegel zeros if $\pi_1,\pi_2,\pi_3$ are all twist-equivalent to each other. 

From now on we may further assume $\pi_3 \cong \pi_1 \otimes \mu$ for some character $\mu$ on $F$, and study Landau-Siegel zeros of $L(s,\Ad(\pi_1) \times \pi_3 \otimes \omega_1\eta)$. In this case we have
\begin{align*}
    L(s,\Ad(\pi_1) \times \pi_3 \otimes \omega_1\eta) &= L(s,\Ad(\pi_1) \times (\pi_1 \otimes \mu) \otimes \omega_1\eta) \\
    &= L(s,(\Ad(\pi_1) \boxtimes \pi_1) \otimes \omega_1\eta\mu) \\
    &= L(s,(\pi_1 \boxplus A^3(\pi_1)) \otimes \omega_1\eta\mu) \\
    &= L(s,\pi_1 \otimes \omega_1\eta\mu) \cdot L(s,A^3(\pi_1) \otimes \omega_1\eta \mu).
\end{align*}
Now $L(s,\pi_1 \otimes \omega_1\eta\mu)$ has no Landau-Siegel zero by Lemma \ref{standard L-function Landau-Siegel zero} (3). To further discuss the Landau-Siegel zeros of $L(s,A^3(\pi_1) \otimes \omega_1\eta\mu)$ we consider the following cases:

\begin{enumerate}
    \item Assume $\pi_1$ is tetrahedral, i.e., $\Ad(\pi_1)$ is cuspidal and monomial. Write $\Ad(\pi_1) \cong \Ad(\pi_1) \otimes \nu$ for some cubic character $\nu$ on $F$. Then by Lemma \ref{symmetric power functorality} (2), $A^3(\pi_1)$ is not cuspidal and decomposes as 
    $$A^3(\pi_1) = (\pi_1 \otimes \nu) \boxplus (\pi_1 \otimes \nu^2).$$
    Thus we have 
    $$L(s,A^3(\pi_1) \otimes \omega_1\eta\mu) = L(s,\pi_1 \otimes \omega_1\eta\mu\nu) \cdot L(s,\pi_1 \otimes \omega_1\eta\mu\nu^2),$$
    which has no Landau-Siegel zero by Lemma \ref{standard L-function Landau-Siegel zero} (3).
    \item Assume $\pi_1$ is octahedral, i.e., $\Sym^3(\pi_1)$ is cuspidal and monomial. Thus $A^3(\pi_1) \otimes \omega_1\eta\mu = \Sym^3(\pi_1) \otimes \eta\mu$ is also cuspidal and monomial. Therefore $L(s,A^3(\pi_1) \otimes \omega_1\eta\mu)$ has no Landau-Siegel zero by Lemma \ref{standard L-function Landau-Siegel zero} (2).
    \item Assume $\pi_1$ is not of solvable polyhedral type. Then $\Sym^3(\pi_1)$ is cuspidal, and so is $\Sym^3(\pi_1) \otimes \eta\mu$. We construct
    $$\Pi := 1 \boxplus \Ad(\pi_1) \boxplus \Sym^3(\pi_1) \otimes \eta\mu.$$
    Then we have 
    \begin{align*}
        L(s,\Pi \times \ov{\Pi}) &= L(s,\Sym^3(\pi_1) \otimes \eta\mu) \cdot L(s,\ov{\Sym^3(\pi_1) \otimes \eta\mu}) \\
        &\cdot L(s,\Ad(\pi_1) \times (\Sym^3(\pi_1) \otimes \eta\mu)) \cdot L(s,\Ad(\pi_1) \times (\ov{\Sym^3(\pi_1) \otimes \eta\mu})) \\
        &\cdot L(s,\Ad(\pi_1))^2 \cdot \zeta_F(s) \cdot L(s,\Ad(\pi_1) \times \Ad(\pi_1)) \\
        &\cdot L(s,(\Sym^3(\pi_1) \otimes \eta\mu) \times (\ov{\Sym^3(\pi_1) \otimes \eta\mu})).
    \end{align*}
    Note that $L(s,\Ad(\pi_1) \times (\Sym^3(\pi_1) \otimes \eta\mu))$ decomposes as 
    \begin{align*}
        L(s,\Ad(\pi_1) \times (\Sym^3(\pi_1) \otimes \eta\mu)) &= L(s,\Sym^3(\pi_1) \otimes \eta\mu) \\
        &\cdot L(s,\pi_1 \otimes \omega_1\eta\mu) \cdot L(s,\pi_1;\Sym^5 \otimes \omega_1^{-1}\eta\mu).
    \end{align*}
    Thus $L(s,\Pi \times \ov{\Pi})$ equals
    \begin{align*}
        L(s,\Pi \times \ov{\Pi}) &= L(s,\Sym^3(\pi_1) \otimes \eta\mu)^2 \cdot L(s,\ov{\Sym^3(\pi_1) \otimes \eta\mu})^2 \\
        &\cdot L(s,\pi_1 \otimes \omega_1\eta\mu) \cdot L(s,\ov{\pi_1 \otimes \omega_1\eta\mu}) \\
        &\cdot L(s,\pi_1;\Sym^5 \otimes \omega_1^{-1}\eta\mu) \cdot L(s,\ov{\pi_1};\Sym^5 \otimes \omega_1\eta^{-1}\mu^{-1}) \\
        &\cdot L(s,\Ad(\pi_1))^2 \cdot \zeta_F(s) \cdot L(s,\Ad(\pi_1) \times \Ad(\pi_1)) \\
        &\cdot L(s,(\Sym^3(\pi_1) \otimes \eta\mu) \times (\ov{\Sym^3(\pi_1) \otimes \eta\mu})).
    \end{align*}
    By Lemma \ref{Landau-Siegel zero lemma}, $L(s,\Sym^3(\pi_1) \otimes \eta\mu)$ has no Landau-Siegel zero if $L(s,\pi_1;\Sym^5 \otimes \omega_1^{-1}\eta\mu)$ is analytic in some interval $(t,1)$, $t<1$. This proves Theorem \ref{Theorem A, refined version} (3).
\end{enumerate}

\subsection{Dihedral Type}\label{Section 3.3}
Assume $(\pi_1,\pi_2,\pi_3)$ is of dihedral type. Recall this means at least one of $\pi_1,\pi_2,\pi_3$, say $\pi_1$, is dihedral. Write $\pi_1 = I_K^F(\chi_1)$ for a character $\chi_1$ on some quadratic extension $K/F$. Then we have 
$$\pi_1 \boxtimes \pi_2 = I_K^F(\chi_1) \boxtimes \pi_2 = I_K^F(B_F^K(\pi_2) \otimes \chi_1).$$
By Lemma \ref{base change and automporphic induction} we have 
\begin{align*}
    L(s,\pi_1 \times \pi_2 \times \pi_3) &= L(s,I_K^F(B_F^K(\pi_2) \otimes \chi_1) \times \pi_3) \\
    &= L(s,B_F^K(\pi_2) \times B_F^K(\pi_3) \otimes \chi_1).
\end{align*}
We discuss Landau-Siegel zeros of the $\GL(2) \times \GL(2)$ Rankin-Selberg $L$-function $L(s,B_F^K(\pi_2) \times B_F^K(\pi_3) \otimes \chi_1)$ by the cuspidality of $B_F^K(\pi_2)$ and $B_F^K(\pi_3)$.

\begin{enumerate}
    \item Assume exactly one of $B_F^K(\pi_2)$ and $B_F^K(\pi_3)$ is cuspidal, say $B_F^K(\pi_2)$ is cuspidal, but $B_F^K(\pi_3)$ is not. By Lemma \ref{base change and automporphic induction} we know that $\pi_3 = I_K^F(\chi_3)$ for some character $\chi_3$ on $K$. In this case we have  
    \begin{align*}
        L(s,\pi_1 \times \pi_2 \times \pi_3) &= L(s,B_F^K(\pi_2) \times B_F^K(\pi_3) \otimes \chi_1) \\
        &= L(s,(B_F^K(\pi_2) \otimes \chi_1) \times (\chi_3 \boxplus \chi_3^\theta)) \\ 
            &= L(s,B_F^K(\pi_2) \otimes \chi_1\chi_3) \cdot L(s,B_F^K(\pi_2) \otimes \chi_1\chi_3^\theta),
    \end{align*}
    which has no Landau-Siegel zero by Lemma \ref{standard L-function Landau-Siegel zero} (3).
    \item Assume that both $B_F^K(\pi_2)$ and $B_F^K(\pi_3)$ are cuspidal. Let $\eta = \eta_{K/F}$ be the quadratic character on $F$ associated with $K/F$. Then by Lemma \ref{base change and automporphic induction} we know that $\pi_2$ is not isomorphic to $\pi_2 \otimes \eta$, and $\pi_3$ is not isomorphic to $\pi_3 \otimes \eta$.
    \begin{enumerate}
        \item Assume further that at least one of $\pi_1 \boxtimes \pi_3$ and $\pi_1 \boxtimes \pi_2$, say $\pi_1 \boxtimes \pi_3$, is cuspidal. Then we have 
        $$(\pi_1 \boxtimes \pi_3) \otimes \eta = (\pi_1 \otimes \eta) \boxtimes \pi_3 = \pi_1 \boxtimes \pi_3.$$
        Since $\pi_2$ is not isomorphic to $\pi_2 \otimes \eta$, the Rankin-Selberg $L$-function
        $$L(s,(\pi_1 \boxtimes \pi_3) \times \pi_2) = L(s,\pi_1 \times \pi_2 \times \pi_3)$$
        has no Landau-Siegel zero by Lemma \ref{new Rankin-Selberg Landau-Siegel zero} (2).
        \item Assume otherwise that neither $\pi_1 \boxtimes \pi_3$ nor $\pi_1 \boxtimes \pi_2$ is cuspidal. Then by Lemma \ref{Rankin-Selberg product functorality} (1), we have $B_F^K(\pi_2) \cong B_F^K(\pi_2) \otimes \chi_1^{\theta}\chi_1^{-1}$ and $B_F^K(\pi_3) \cong B_F^K(\pi_3) \otimes \chi_1^\theta\chi_1^{-1}$. Let $L/K$ be the quadratic extension associated with the quadratic character $\chi_1^\theta\chi_1^{-1}$ on $K$. Then $B_F^K(\pi_2) \otimes \chi_1 = I_L^K(\xi_2)$ and $B_F^K(\pi_3) = I_L^K(\xi_3)$ for characters $\xi_2,\xi_3$ on $L$. In this case we have 
        \begin{align*}
            L(s,\pi_1 \times \pi_2 \times \pi_3) &= L(s,B_F^K(\pi_2) \times B_F^K(\pi_3) \otimes \chi_1) \\
            &= L(s,I_L^K(\xi_2) \times I_L^K(\xi_3)) \\
            &= L(s,\xi_3\xi_2) \cdot L(s,\xi_3\xi_2^{\theta_{L/K}}).
        \end{align*}
        Thus $L(s,\pi_1 \times \pi_2 \times \pi_3)$ has no Landau-Siegel zero unless $(\xi_3\xi_2)^2 = 1$ or $(\xi_3\xi_2^{\theta_{L/K}})^2 =1$. Moreover, there is at most one Landau-Siegel zero, by the same argument as in the proof of \cite[Theorem A]{RamakrishnanWang2003}. This proves Theorem \ref{Theorem A, refined version} (2).
    \end{enumerate}
    \item Assume that neither $B_F^K(\pi_2)$ nor $B_F^K(\pi_3)$ is cuspidal. Then $\pi_2 = I_K^F(\chi_2)$ and $\pi_3 = I_K^F(\chi_3)$ for characters $\chi_2,\chi_3$ on $F$. By Lemma \ref{base change and automporphic induction} we have $B_F^K(\pi_2) = \chi_2 \boxplus \chi_2^\theta$ and $B_F^K(\pi_3) = \chi_3 \boxplus \chi_3^\theta$. In this case we have 
    \begin{align*}
        L(s,\pi_1 \times \pi_2 \times \pi_3) &= L(s,(\chi_2 \boxplus \chi_2^\theta) \times (\chi_3 \boxplus \chi_3^\theta) \otimes \chi_1) \\
        &= L(s,\chi_1\chi_2\chi_3)\cdot L(s,\chi_1\chi_2\chi_3^\theta) \cdot L(s,\chi_1\chi_2^\theta\chi_3) \cdot L(s,\chi_1\chi_2^\theta\chi_3^\theta).
    \end{align*}
    A priori there are at most four Landau-Siegel zeros, if the four characters $\chi_1\chi_2\chi_3, \chi_1\chi_2\chi_3^\theta, \chi_1\chi_2^\theta\chi_3$ and $\chi_1\chi_2^\theta\chi_3^\theta$ are all real. However, we can generalize the argument in \cite[Section 4]{RamakrishnanWang2003} to show that there are at most two (and in most cases, even at most one) Landau-Siegel zeros. We discuss the following cases depending on the number of real characters among $\chi_1\chi_2\chi_3, \chi_1\chi_2\chi_3^\theta, \chi_1\chi_2^\theta\chi_3$ and $\chi_1\chi_2^\theta\chi_3^\theta$.
    \begin{enumerate}
        \item Assume there is at most one real character. Then $L(s,\pi_1 \times \pi_2 \times \pi_3)$ has at most one Landau-Siegel zero, since $L$-functions of complex characters have no Landau-Siegel zero by Lemma \ref{standard L-function Landau-Siegel zero} (1).
        \item Assume there are exactly two real characters, say $\chi_1\chi_2\chi_3$ and $\chi_1\chi_2\chi_3^\theta$ are real. Then one can show that $L(s,\chi_1\chi_2\chi_3) \cdot L(s,\chi_1\chi_2\chi_3^\theta)$ has at most one Landau-Siegel zero using the same argument as in \cite[Section 4]{RamakrishnanWang2003}.
        \item We claim that there cannot be exactly three real characters. This is by noting that if three of them, say $\chi_1\chi_2\chi_3,\chi_1\chi_2\chi_3^\theta$ and $\chi_1\chi_2^\theta\chi_3$, are real, then the fourth character 
        $$\chi_1\chi_2^{\theta}\chi_3^{\theta} = \frac{\chi_1\chi_2\chi_3^{\theta} \cdot \chi_1\chi_2^{\theta}\chi_3 }{\chi_1\chi_2\chi_3}$$
        is also real.
        \item Assume that all four characters are real. Then we further discuss by counting the number of trivial characters.
        \begin{enumerate}
            \item Assume further that there is at least one trivial character, say $\chi_1\chi_2\chi_3 = 1$. We set $\mu := \chi_1\chi_2\chi_3^\theta$ and $\nu := \chi_1\chi_2^\theta\chi_3$. Then 
            $$\chi_1\chi_2^{\theta}\chi_3^{\theta} = \frac{\chi_1\chi_2\chi_3^{\theta} \cdot \chi_1\chi_2^{\theta}\chi_3 }{\chi_1\chi_2\chi_3} = \mu\nu.$$
            In this case we have 
            $$L(s,\pi_1 \times \pi_2 \times \pi_3) = \zeta_K(s) L(s,\mu)L(s,\nu)L(s,\mu\nu).$$
            Note that since $\chi_2 \neq \chi_2^\theta$ and $\chi_3 \neq \chi_3^\theta$, $\mu$ and $\nu$ are both non-trivial quadratic. Thus $L(s,\pi_1 \times \pi_2 \times \pi_3)$ is a Dirichlet series with non-negative coefficients, that has a pole of order $2$ or $1$ at $s=1$, depending on whether $\mu\nu = 1$ or not (equivalently, whether $(\chi_2\chi_3)^\theta = \chi_2\chi_3$ or not). Thus $L(s,\pi_1 \times \pi_2 \times \pi_3)$ has at most two Landau-Siegel zeros (or at most one Landau-Siegel zero, if $(\chi_2\chi_3)^\theta \neq \chi_2\chi_3$).
            \item Assume otherwise that all four characters are non-trivial quadratic. We set $\alpha := \chi_1\chi_2\chi_3$, $\beta := \chi_3^\theta/\chi_3$ and $\gamma := \chi_2^\theta/\chi_2$. Note that by our assumption $\alpha,\beta,\gamma$ are all non-trivial quadratic. We have the following relations:
            \begin{align*}
                    \chi_1\chi_2\chi_3^{\theta} &= \alpha\beta, \\
                    \chi_1\chi_2^{\theta}\chi_3 &= \alpha\gamma, \\
                    \chi_1\chi_2^{\theta}\chi_3^{\theta} &= \alpha\beta\gamma.
            \end{align*}
            In this case we have 
            $$L(s,\pi_1 \times \pi_2 \times \pi_3) = L(s,\alpha)L(s,\alpha\beta)L(s,\alpha\gamma)L(s,\alpha\beta\gamma).$$
            Consider the auxiliary $L$-function
            $$L(s) := \zeta_{K}(s)L(s,\alpha)L(s,\beta)L(s,\gamma)L(s,\alpha\beta)L(s,\alpha\gamma)L(s,\beta\gamma)L(s,\alpha\beta\gamma),$$
            which is a Dirichlet series with non-negative coefficients, that has a pole of order $2$ or $1$ at $s=1$, depending on whether $\beta\gamma = 1$ or not (equivalently, whether $(\chi_2\chi_3)^\theta = \chi_2\chi_3$ or not). Thus $L(s,\pi_1 \times \pi_2 \times \pi_3)$, as part of $L(s)$, has at most two Landau-Siegel zeros (or at most one Landau-Siegel zero, if $(\chi_2\chi_3)^\theta \neq \chi_2\chi_3$). We have thus finished the proof of Theorem \ref{Theorem A, refined version} (1).
        \end{enumerate}
    \end{enumerate}
\end{enumerate}

\section{Proof of Theorem B}\label{Section 4}

The goal of this section is to prove the following refined version of Theorem \ref{Theorem B}:

\begin{theorem}\label{Theorem B, refined version}
    Let $\pi_1,\pi_2 \in \cA_0(2,F)$ and $\pi \in \cA_0(3,F)$ with central characters $\omega_1,\omega_2$ and $\omega$ respectively. Then $L(s,\pi_1 \times \pi_2 \times \pi)$ has no Landau-Siegel zero except possibly in the following cases:
    \begin{enumerate}
        \item $\pi_1,\pi_2$ are non-dihedral twist-equivalent, and $\pi$ is twist-equivalent to $\Ad(\pi_1)$. Moreover, write $\pi_2 \cong \pi_1 \otimes \eta$ and $\pi \otimes \omega_1\eta \cong \Ad(\pi_1) \otimes \mu$ for character $\eta$ and $\mu$ on $F$. Then there is at most one Landau-Siegel zero coming form $L(s,\mu)$, which can only occur if $\mu^2 = 1$.
        \item $\pi_1,\pi_2$ are not of solvable polyhedral type, and are not twist-equivalent, but $\pi$ is twist-equivalent to either $\Ad(\pi_1)$ or $\Ad(\pi_2)$, say $\Ad(\pi_1)$. Moreover, write $\pi \cong \Ad(\pi_1) \otimes \eta$ for some character $\eta$ on $F$. Then there is no Landau-Siegel zero if $A^3(\pi_2 \otimes \eta)$ is not isomorphic to $\ov{A^3(\pi_1)}$. 
    \end{enumerate}
\end{theorem}

\begin{remark}
    The condition in Theorem \ref{Theorem B, refined version} (2) is satisfied if $F$ is totally real, and $\pi_1,\pi_2$ are regular algebraic by \cite[Lemma 4.2]{Thorner2025}. Thus there is no Landau-Siegel zero in this case.
\end{remark}

\subsection{General Type}
Assume $(\pi_1,\pi_2,\pi)$ is of general type. Recall that this means $\pi_1,\pi_2$ are non-dihedral, not twist-equivalent, and $\pi$ is not twist-equivalent to either $\Ad(\pi_1)$ or $\Ad(\pi_2)$. We set 
$$\Pi := (\pi_1 \boxtimes \pi) \boxplus \ov{\pi_2} \boxplus (\Ad(\pi_1) \boxtimes \ov{\pi_2}).$$
Now since $\pi_1$ is non-dihedral and $\pi$ is not twist-equivalent to $\Ad(\pi_1)$, by Lemma \ref{Rankin-Selberg product functorality} (2), we know that $\pi_1 \boxtimes \pi$ is cuspidal. Moreover, since $\pi_1$ and $\pi_2$ are not twist-equivalent, by Lemma \ref{multiplicity one} we know that $\Ad(\pi_1)$ and $\Ad(\ov{\pi_2}) = \Ad(\pi_2)$ are not twist-equivalent either. Therefore $\Ad(\pi_1) \boxtimes \ov{\pi_2}$ is also cuspidal. We then claim that $\pi_1 \boxtimes \pi$ and $\Ad(\pi_1) \boxtimes \ov{\pi_2}$ are not isomorphic. To do this we compute the Rankin-Selberg $L$-function $L(s,(\pi_1 \boxtimes \pi) \times \ov{(\Ad(\pi_1) \boxtimes \ov{\pi_2})})$:
\begin{align*}
    L(s,(\pi_1 \boxtimes \pi) \times \ov{(\Ad(\pi_1) \boxtimes \ov{\pi_2})})&= L(s,(\pi_1 \boxtimes \pi) \times (\Ad(\pi_1) \boxtimes \pi_2)) \\
    &= L(s,(\pi_1 \boxtimes \Ad(\pi_1)) \times (\pi_2 \boxtimes \pi)) \\
    &= L(s,(\pi_1 \boxplus A^3(\pi_1)) \times (\pi_2 \boxtimes \pi)) \\
    &= L(s,\pi_1 \times (\pi_2 \boxtimes\pi)) \cdot L(s,A^3(\pi_1) \times (\pi_2 \boxtimes \pi)).
\end{align*}
Now since $\pi_2$ is non-dihedral and $\pi$ is not twist-equivalent to $\Ad(\pi_2)$, by Lemma \ref{Rankin-Selberg product functorality} (2), $\pi_2 \boxtimes \pi$ is cuspidal of degree $6$, which is not isomorphic to either $\ov{\pi_1}$ or any isobaric constitute of $\ov{A^3(\pi_1)}$. Therefore both $L(s,\pi_1 \times (\pi_2 \boxtimes \pi))$ and $L(s,A^3(\pi_1) \times (\pi_2 \boxtimes \pi))$ are entire. Thus $\pi_1 \boxtimes \pi$ is not isomorphic to $\Ad(\pi_1) \boxtimes \ov{\pi_2}$. In view of the decomposition above we have
\begin{align*}
    L(s,\Pi \times \ov{\Pi}) &= L(s,\pi_1 \times \pi_2 \times \pi)^2 \cdot L(s,\ov{\pi_1} \times \ov{\pi_2} \times \ov{\pi})^2 \\
    &\cdot L(s,A^3(\pi_1) \times (\pi_2 \boxtimes \pi)) \cdot L(s,\ov{A^3(\pi_1)} \times (\ov{\pi_2 \boxtimes \pi})) \\
    &\cdot L(s,\pi_2 \times (\Ad(\pi_1) \boxtimes \ov{\pi_2}))^2 \\
    &\cdot L(s,(\pi_1 \boxtimes \pi) \times (\ov{\pi_1 \boxtimes \pi})) \cdot L(s,\pi_2 \times \ov{\pi_2}) \cdot L(s,(\Ad(\pi_1) \boxtimes \pi_2) \times (\ov{\Ad(\pi_1) \boxtimes \pi_2})).
\end{align*}
By Lemma \ref{Landau-Siegel zero lemma}, $L(s,\pi_1 \times \pi_2 \times \pi)$ has no Landau-Siegel zero.

\subsection{Twist-equivalent Type}
Assume $(\pi_1,\pi_2,\pi)$ is of twist-equivalent type. Recall that this means $\pi_1,\pi_2$ are both non-dihedral, but either (1) $\pi_1$ is twist-equivalent to $\pi_2$, or (2) $\pi$ is twist-equivalent to $\Ad(\pi_1)$ or $\Ad(\pi_2)$. Thus we have two cases to consider:

\begin{enumerate}
    \item Assume that $\pi_1$ is twist-equivalent to $\pi_2$. Then $\pi_2 \cong \pi_1 \otimes \eta$ for some character $\eta$ on $F$. In this case $L(s,\pi_1 \times \pi_2 \times \pi)$ decomposes as 
    \begin{align*}
        L(s,\pi_1 \times \pi_2 \times \pi) &= L(s,\pi_1 \times (\pi_1 \otimes \eta) \times \pi) \\
        &= L(s,(\pi_1 \boxtimes \pi_1) \times \pi \otimes \eta) \\
        &= L(s,(\Ad(\pi_1) \otimes \omega_1 \boxplus \omega_1) \times \pi \otimes \eta) \\
        &= L(s,\pi \otimes \omega_1\eta) \cdot L(s,\Ad(\pi_1) \times (\pi \otimes \omega_1\eta)).
    \end{align*}
    Now $L(s,\pi \otimes \omega_1\eta)$ has no Landau-Siegel zero by Lemma \ref{standard L-function Landau-Siegel zero} (3). If $\pi \otimes \omega_1\eta$ is non-self-dual, then $L(s,\Ad(\pi_1) \times (\pi \otimes \omega_1\eta))$ has no Landau-Siegel zero by Lemma \ref{new Rankin-Selberg Landau-Siegel zero} (1). Therefore $L(s,\pi_1 \times \pi_2 \times \pi)$ has no Landau-Siegel zero in this case. If $\pi \otimes \omega_1\eta$ is self-dual, then by \cite[Theorem A]{Ramakrishnan2014} we have $\pi \otimes \omega_1\eta \cong \Ad(\pi_3) \otimes \mu$ for some non-dihedral $\pi_3 \in \cA_0(2,F)$ and some character $\mu$ on $F$ with $\mu^2 = 1$. We are thus reduced to study Landau-Siegel zeros of the Rankin-Selberg $L$-function
    $$L(s,\Ad(\pi_1) \times \Ad(\pi_3) \otimes \mu),$$
    which has been studied in \cite{Thorner2025symmetricsquare}. We will briefly record the proof here for completeness. 
    \begin{enumerate}
        \item Assume further that $\pi_1$ is not twist-equivalent to $\pi_3$. Consider the following auxiliary $L$-function
        \begin{align*}
            L(s) &:= L(s,\Ad(\pi_1) \times \Ad(\pi_3) \otimes \mu)^8 \\
            &\cdot L(s,A^4(\pi_1) \times \Ad(\pi_3) \otimes \mu)^4 \cdot L(s,\Ad(\pi_3) \otimes \mu)^4 \\
            &\cdot L(s,\Ad(\pi_1) \times A^4(\pi_3))^2 \cdot L(s,\Ad(\pi_1) \times \Ad(\pi_3))^2 \cdot L(s,\Ad(\pi_1))^2 \\
            &\cdot L(s,\Ad(\pi_1) \times \Ad(\pi_1))^4 \cdot L(s,\Ad(\pi_3) \times \Ad(\pi_3))\\
            &\cdot L(s,(\Ad(\pi_1) \boxtimes \Ad(\pi_1)) \times (\Ad(\pi_3) \boxtimes \Ad(\pi_3))).
        \end{align*}
        Let us write 
        \begin{align*}
            L(s) &= \prod_v\prod_{j=1}^{d_L} \left(1-\frac{\alpha_{L,j}(v)}{q_v^s}\right)^{-1}, \\
            -\frac{L^\prime}{L}(s) &= \sum_{v} \sum_{\ell=0}^\infty \frac{a_L(v^\ell)}{q_v^{\ell s}}\cdot\log (q_v),
        \end{align*}
        where $d_L$ is the degree of $L(s)$, $v$ runs through finite places of $F$, and $q_v = \#F_v/\cO_v$ is the order of the residue field at $v$. The following two facts can be verified for $L(s)$:
        \begin{enumerate}
            \item For each finite place $v$ where $\pi_1,\pi_3$ and $\mu$ are unramified, we have 
            $$a_L(v^\ell) = |2a_{\Ad(\pi_1)}(v^\ell) + a_{\Ad(\pi_3) \otimes \mu}(v^\ell) + a_{\Ad(\pi_1) \times \Ad(\pi_3) \otimes \mu}(v^\ell)|^2 \geq 0.$$ 
            \item For each finite place $v$, we have $\alpha_{L,j}(v) \ll q_v^{\frac{8}{9}}$.
        \end{enumerate}
        The fact (i) is a direct computation and can be found in \cite[Lemma 3.2]{Thorner2025symmetricsquare}. The fact (ii) is due to the bound $|\alpha_{\pi_1,j}(v)|,|\alpha_{\pi_3,j}(v)| \leq  q_v^{\frac{1}{9}}$ \cite[Proposition 4.8]{KimShahidi2002FourthPowerCuspidality}. Thus by \cite[Lemma 5.9]{IwaniecKowalskiANT}, $L(s)$ has at most $r$ Landau-Siegel zeros, where $r$ is the order of pole of $L(s)$ at $s=1$. Let us proceed to count the order of pole of $L(s)$ at $s=1$, and compare it to the multiplicity of the factor $L(s,\Ad(\pi_1) \times \Ad(\pi_3) \otimes \mu)$. To do this we first apply the decomposition
        $$\Ad(\pi_1) \boxtimes \Ad(\pi_1) = A^4(\pi_1) \boxplus \Ad(\pi_1) \boxplus 1$$
        to further write $L(s)$ as 
        \begin{align*}
            L(s) &:= L(s,\Ad(\pi_1) \times \Ad(\pi_3) \otimes \mu)^8 \\
            &\cdot L(s,A^4(\pi_1) \times \Ad(\pi_3) \otimes \mu)^4 \cdot L(s,\Ad(\pi_3) \otimes \mu)^4 \\
            &\cdot L(s,\Ad(\pi_1) \times A^4(\pi_3))^3 \cdot L(s,\Ad(\pi_1) \times \Ad(\pi_3))^3 \cdot L(s,\Ad(\pi_1))^3 \\
            &\cdot L(s,A^4(\pi_1) \times A^4(\pi_3)) \cdot L(s,A^4(\pi_1) \times \Ad(\pi_3)) \cdot L(s,A^4(\pi_1)) \\
            &\cdot L(s,A^4(\pi_3)) \cdot L(s,\Ad(\pi_3)) \\
            &\cdot L(s,\Ad(\pi_1) \times \Ad(\pi_1))^4 \cdot L(s,\Ad(\pi_3) \times \Ad(\pi_3)) \cdot \zeta_F(s).
        \end{align*}
        Note that the last line of the expression above contributes to $L(s)$ a pole at $s=1$ of order 6. The only other factors that could possibly contribute to $L(s)$ a pole at $s=1$ are $L(s,A^4(\pi_1) \times A^4(\pi_3)), L(s,\Ad(\pi_1) \times A^4(\pi_3))^3, L(s,A^4(\pi_1) \times \Ad(\pi_3))$ and $L(s,A^4(\pi_1) \times \Ad(\pi_3) \otimes \mu)^4$. However, in view of Lemma \ref{symmetric power functorality} (3), Lemma \ref{multiplicity one} and the assumption that $\pi_1$ is not twist-equivalent to $\pi_3$, $L(s,\Ad(\pi_1) \times A^4(\pi_3))^3, L(s,A^4(\pi_1) \times \Ad(\pi_3))$ and $L(s,A^4(\pi_1) \times \Ad(\pi_3) \otimes \mu)^4$ are entire. Thus the only factor that my contribute to $L(s)$ additional poles at $s=1$ is $L(s,A^4(\pi_1) \times A^4(\pi_3))$.
        \begin{enumerate}
            \item Assume that $\pi_1$ is not of solvable polyhedral type. Then $A^4(\pi_1)$ is cuspidal. In this case the order of pole of $L(s)$ at $s=1$ is 7 or 6, depending on whether $A^4(\pi_1) \cong A^4(\pi_3)$ or not. In either case $L(s,\Ad(\pi_1) \times \Ad(\pi_3) \otimes \mu)$ has no Landau-Siegel zero, since any Landau-Siegel zero of it would be a Landau-Siegel zero of order at least 8 of $L(s).$
            \item Assume that $\pi_1$ is of octahedral type. Then by Lemma \ref{symmetric power functorality} (3), $A^4(\pi_1)$ decomposes as 
            $$A^4(\pi_1) = (Ad(\pi_1) \otimes \nu) \boxplus I_K^F(\chi^{-1}) $$
            for some quadratic character $\nu$ on $F$ and some character $\chi$ on a quadratic extension $K/F$. In this case $L(s,A^4(\pi_1) \times A^4(\pi_3))$ has at most 1 simple pole at $s=1$, which only occurs if $\pi_3$ is also octahedral. Thus $r \leq 7 < 8$ and $L(s,\Ad(\pi_1) \times \Ad(\pi_3) \otimes \mu)$ has no Landau-Siegel zero.
            \item Assume that $\pi_1$ is of tetrahedral type. Then by Lemma \ref{symmetric power functorality} (3), $A^4(\pi_1)$ decomposes as 
            $$A^4(\pi_1) = \Ad(\pi_1) \boxplus \nu \boxplus \nu^2$$
            for some cubic character $\nu$ on $F$. In this case $L(s,A^4(\pi_1) \times A^4(\pi_3))$ can only have a pole at $s=1$ if $\pi_3$ is also tetrahedral, and the order of the pole is at most 4. We then make use of the decomposition 
            \begin{align*}
                L(s,A^4(\pi_1) \times \Ad(\pi_3) \otimes \mu) &= L(s,(\Ad(\pi_1) \boxplus \nu \boxplus \nu^2) \times \Ad(\pi_3) \otimes \mu) \\
                &= L(s,\Ad(\pi_3) \otimes \mu\nu) \cdot L(s,\Ad(\pi_3) \otimes \mu\nu^2) \\
                &\cdot L(s,\Ad(\pi_1) \times \Ad(\pi_3) \otimes \mu)
            \end{align*}
            to see that any Landau-Siegel zero of $L(s,\Ad(\pi_1) \times \Ad(\pi_3) \otimes \mu)$ would be a Landau-Siegel zero of $L(s)$ of order at least 12. Since $r \leq 10 < 12$, $L(s,\Ad(\pi_1) \times \Ad(\pi_3) \otimes \mu)$ has no Landau-Siegel zero.
        \end{enumerate}
        \item Assume otherwise that $\pi_1$ is twist-equivalent to $\pi_3$. In this case $L(s,\Ad(\pi_1) \times \Ad(\pi_3) \otimes \mu)$ decomposes as 
        \begin{align*}
            L(s,\Ad(\pi_1) \times \Ad(\pi_3) \otimes \mu) &= L(s,\Ad(\pi_1) \boxtimes \Ad(\pi_1) \otimes \mu) \\
            &= L(s,(A^4(\pi_1) \boxplus \Ad(\pi_1) \boxplus 1) \otimes \mu) \\
            &= L(s,A^4(\pi_1) \otimes \mu) \cdot L(s,\Ad(\pi_1) \otimes \mu) \cdot L(s,\mu).
        \end{align*}
        Now $L(s,\Ad(\pi_1) \otimes \mu)$ has no Landau-Siegel zero by Lemma \ref{standard L-function Landau-Siegel zero} (3). We then prove that 
        $L(s,A^4(\pi_1) \otimes \mu)$ has no Landau-Siegel zero.
        \begin{enumerate}
            \item Assume that $\pi_1$ is of tetrahedral type. Then by Lemma \ref{symmetric power functorality} (3), $L(s,A^4(\pi_1) \otimes \mu)$ decomposes as 
            \begin{align*}
                L(s,A^4(\pi_1) \otimes \mu) &= L(s,(\Ad(\pi_1) \boxplus \nu \boxplus \nu^2) \otimes \mu) \\
                &= L(s,\Ad(\pi_1) \otimes \mu) \cdot L(s,\mu\nu) \cdot L(s,\mu\nu^2)
            \end{align*}
            for some cubic character $\nu$ on $F$. Now $L(s,\Ad(\pi_1) \otimes \mu)$ has no Landau-Siegel zero by Lemma \ref{standard L-function Landau-Siegel zero} (3). Moreover, $L(s,\mu\nu)$ and $L(s,\mu\nu^2)$ have no Landau-Siegel zero since $(\mu\nu)^2 = \nu^2 \neq 1$ and $(\mu\nu^2)^2 = \nu \neq 1$. Therefore $L(s,A^4(\pi_1) \otimes \mu)$ has no Landau-Siegel zero.
            \item Assume that $\pi_1$ is of octahedral type. Then by Lemma \ref{symmetric power functorality} (3), $L(s,A^4(\pi_1) \otimes \mu)$ decomposes as 
            \begin{align*}
                L(s,A^4(\pi_1) \otimes \mu) &= L(s,((\Ad(\pi_1) \otimes \nu) \boxplus I_K^F(\chi^{-1})) \otimes \mu) \\
                &= L(s,\Ad(\pi_1) \otimes \mu\nu) \cdot L(s,I_K^F(\chi^{-1}) \otimes \mu)
            \end{align*}
            for some quadratic character $\nu$ on $F$ and some character $\chi$ on a quadratic extension $K/F$, which has no Landau-Siegel zero by Lemma \ref{standard L-function Landau-Siegel zero} (3).
            \item Assume that $\pi_1$ is not of solvable polyhedral type. Then $A^4(\pi_1)$ is cuspidal. We set 
            $$\Pi := 1 \boxplus (A^4(\pi_1) \otimes \mu) \boxplus \Ad(\pi_1)$$
            and compute 
            \begin{align*}
                L(s,\Pi \times \ov{\Pi}) &= L(s,A^4(\pi_1) \otimes \mu)^2 \\
                &\cdot L(s,(A^4(\pi_1) \otimes \mu)\times \Ad(\pi_1))^2 \cdot L(s,\Ad(\pi_1))^2 \\
                &\cdot \zeta_F(s) \cdot L(s,\Ad(\pi_1) \times \Ad(\pi_1)) \cdot L(s,A^4(\pi_1) \times A^4(\pi_1))
            \end{align*}
            Note that $L(s,(A^4(\pi_1) \otimes \mu)\times \Ad(\pi_1))$ decomposes as 
            \begin{align*}
                L(s,(A^4(\pi_1) \otimes \mu)\times \Ad(\pi_1)) &= L(s,A^4(\pi_1) \otimes \mu) \\
                &\cdot L(s,\pi_1;\Sym^6 \otimes \omega_1^{-3}\mu) \cdot L(s,\Ad(\pi_1) \otimes \mu) \\
                &=L(s,A^4(\pi_1) \otimes \mu) \cdot L(s,\Sym^3(\pi_1);\Sym^2\otimes \omega_1^{-3}\mu)
            \end{align*}
            Thus $L(s,\Pi \times \ov{\Pi})$ can be further written as 
            \begin{align*}
                L(s,\Pi \times \ov{\Pi}) &= L(s,A^4(\pi_1) \otimes \mu)^4 \\
                &\cdot L(s,\Sym^3(\pi_1);\Sym^2\otimes \omega_1^{-3}\mu)^2 \cdot L(s,\Ad(\pi_1))^2 \\
                &\cdot \zeta_F(s) \cdot L(s,\Ad(\pi_1) \times \Ad(\pi_1)) \cdot L(s,A^4(\pi_1) \times A^4(\pi_1)).
            \end{align*}
            Let $S$ be the set of finite places of $F$ where $\pi_1$ or $\mu$ is ramified. By \cite[Corollary 5.8]{Takeda2014}, the partial $L$-function $L^S(s,\Sym^3(\pi_1);\Sym^2\otimes \omega_1^{-3}\mu)$ is holomorphic in $(7/8,1)$. Moreover, the finite product of local $L$-functions 
            $$\prod_{v \in S}L(s,\Sym^3(\pi_{1,v});\Sym^2 \otimes \omega_{1,v}1^{-3}\mu_v) = \prod_{v \in S} \frac{L(s,A^4(\pi_{1,v}) \times \Ad(\pi_{1,v}) \otimes \mu_v)}{L(s,A^4(\pi_{1,v}) \otimes \mu_v)}$$
            is holomorphic in $(56/65,1)$ by \cite[Theorem 2]{LuoRudnickSarnak1999}.
            Thus by Lemma \ref{Landau-Siegel zero lemma}, $L(s,A^4(\pi_1) \otimes \mu)$ has no Landau-Siegel zero.
        \end{enumerate}
        We have thus eliminated Landau-Siegel zeros of $L(s,A^4(\pi_1) \otimes \mu)$ in all cases. Therefore the only possible Landau-Siegel zero of $L(s,\Ad(\pi_1) \times \Ad(\pi_3) \otimes \mu)$ comes from $L(s,\mu)$. This proves Theorem \ref{Theorem B, refined version} (1).
    \end{enumerate}
    \item Assume $\pi_1,\pi_2$ are not twist-equivalent, and that $\pi$ is twist-equivalent to $\Ad(\pi_1)$ or $\Ad(\pi_2)$, say $\pi \cong \Ad(\pi_1) \otimes \eta$ for some character $\eta$ on $F$. In this case $L(s,\pi_1 \times \pi_2 \times \pi)$ decomposes as 
    \begin{align*}
        L(s,\pi_1 \times \pi_2 \times \pi) &= L(s,\pi_1 \times \pi_2 \times (\Ad(\pi_1) \otimes \eta)) \\
        &= L(s,(\pi_1 \boxtimes \Ad(\pi_1)) \times \pi_2 \otimes \eta) \\
        &= L(s,(\pi_1 \boxplus A^3(\pi_1)) \times \pi_2 \otimes \eta) \\
        &= L(s,\pi_1 \times \pi_2 \otimes \eta) \cdot L(s,A^3(\pi_1) \times \pi_2 \otimes \eta).
    \end{align*}
    Now since $\pi_1,\pi_2$ are non-dihedral and are not twist-equivalent, $L(s,\pi_1 \times \pi_2 \otimes \eta)$ has no Landau-Siegel zero by Lemma \ref{Rankin Selberg Landau-Siegel zero} (1). We then discuss Landau-Siegel zeros of $L(s,A^3(\pi_1) \times \pi_2 \otimes \eta)$:
    \begin{enumerate}
        \item Assume $\pi_1$ is tetrahedral. Then by Lemma \ref{symmetric power functorality} (2), $A^3(\pi_1)$ decomposes as 
        $$A^3(\pi_1) = (\pi_1 \otimes \nu) \boxplus (\pi_1 \otimes \nu^2)$$
        for some cubic character $\nu$ on $F$. Thus we have 
        $$L(s,A^3(\pi_1) \times \pi_2 \otimes \eta) = L(s,\pi_1 \times \pi_2 \otimes \eta\nu) \cdot L(s,\pi_1 \times \pi_2 \otimes \eta\nu^2)$$
        which has no Landau-Siegel by Lemma \ref{Rankin Selberg Landau-Siegel zero} (1).
        \item Assume $\pi_1$ is octahedral, i.e., $A^3(\pi_1)$ is cuspidal and monomial. Write $A^3(\pi_1) \cong A^3(\pi_1) \otimes \mu$ for some quadratic character $\mu$ on $F$. Since $\pi_2$ is non-dihedral, $\pi_2 \otimes \eta$ is not isomorphic to $(\pi_2 \otimes \eta) \otimes \mu$. Therefore $L(s,A^3(\pi_1) \times \pi_2 \otimes \eta)$ has no Landau-Siegel zero by Lemma \ref{new Rankin-Selberg Landau-Siegel zero} (2).
        \item Assume $\pi_1$ is not of solvable polyhedral type. In this case we go back to the triple product $L$-function
        $$L(s,\pi_1 \times \pi_2 \times \pi) = L(s,\pi_1 \times \pi_2^\prime \times \Ad(\pi_1)).$$
        Here we write $\pi_2^\prime = \pi_2 \otimes \eta$. We set 
        $$\Pi := (\pi_2^\prime \boxtimes \Ad(\pi_1)) \boxplus \ov{\pi_1} \boxplus (\ov{\pi_1} \boxtimes \Ad(\pi_2^\prime)).$$
        To test whether $\pi_2^\prime \boxtimes \Ad(\pi_1)$ and $\ov{\pi_1} \boxtimes \Ad(\pi_2^\prime)$ are isomorphic we compute
        \begin{align*}
            L(s,(\pi_2^\prime \boxtimes \Ad(\pi_1)) \times (\pi_1 \boxtimes \Ad(\pi_2^\prime))) &= L(s,(\pi_1 \boxtimes \Ad(\pi_1)) \times (\pi_2^\prime \boxtimes \Ad(\pi_2^\prime))) \\
            &= L(s,(\pi_1 \boxplus A^3(\pi_1)) \times (\pi_2^\prime \boxplus A^3(\pi_2^\prime))) \\
            &= L(s,\pi_1 \times \pi_2^\prime) \cdot L(s,A^3(\pi_1) \times A^3(\pi_2^\prime)) \\
            &\cdot L(s,\pi_1 \times A^3(\pi_2^\prime)) \cdot L(s,\pi_2^\prime \times A^3(\pi_1)),
        \end{align*}
        which can only have a simple pole at $s=1$ if $A^3(\pi_2^\prime) \cong \ov{A^3(\pi_1)}$, which would imply that $\pi_2$ is not of solvable polyhedral type. By the decomposition above we have 
        \begin{align*}
            L(s,\Pi \times \ov{\Pi}) &= L(s,\pi_1 \times \pi_2^\prime \times \Ad(\pi_1))^2 \cdot L(s,\ov{\pi_1} \times \ov{\pi_2^\prime} \times \Ad(\pi_1))^2 \\
            &\cdot L(s,\pi_1 \times A^3(\pi_2^\prime)) \cdot L(s,\ov{\pi_1} \times \ov{A^3(\pi_2^\prime)}) \\
            &\cdot L(s,A^3(\pi_1) \times A^3(\pi_2^\prime)) \cdot L(s,\ov{A^3(\pi_1)} \times \ov{A^3(\pi_2^\prime)}) \\
            &\cdot L(s,\pi_1 \times (\ov{\pi_1} \boxtimes \Ad(\pi_2^\prime)))^2 \\
            &\cdot L(s,(\pi_2^\prime \boxtimes \Ad(\pi_1)) \times (\ov{\pi_2^\prime} \boxtimes \Ad(\pi_1))) \cdot L(s,\pi_1 \times \ov{\pi_1}) \\
            &\cdot L(s,(\pi_1 \boxtimes \Ad(\pi_2^\prime)) \times (\ov{\pi_1} \boxtimes \Ad(\pi_2^\prime))).
        \end{align*}
        Thus by Lemma \ref{Landau-Siegel zero lemma}, $L(s,\pi_1 \times \pi_2 \times \pi)$ can only have Landau-Siegel zeros if $\pi_2$ is not of solvable polyhedral type, and that 
        $$A^3(\pi_2 \otimes \eta) \cong \ov{A^3(\pi_1)}.$$
        This proves Theorem \ref{Theorem B, refined version} (2).
    \end{enumerate}
\end{enumerate}

\subsection{Dihedral Type}
Assume $(\pi_1,\pi_2,\pi)$ is of dihedral type. Recall that this means $\pi_1$ or $\pi_2$, say $\pi_1$, is dihedral. Write $\pi_1 = I_K^F(\chi_1)$ for a character $\chi_1$ on some quadratic extension $K/F$. Then we have 
\begin{align*}
    L(s,\pi_1 \times \pi_2 \times \pi) &= L(s,(I_K^F(\chi_1) \boxtimes \pi_2) \times \pi) \\
    &= L(s,I_K^F(B_F^K(\pi_2) \otimes \chi_1) \times \pi) \\
    &= L(s,B_F^K(\pi_2) \times B_F^K(\pi) \otimes \chi_1).
\end{align*}
Let $\eta = \eta_{K/F}$ be the quadratic character on $F$ associated with $K/F$. Let $\theta = \theta_{K/F}$ be the non-trivial element in Gal$(K/F)$. Then $\pi$ is not isomorphic to $\pi \otimes \eta$. By Lemma \ref{base change and automporphic induction} we know that $B_F^K(\pi)$ is cuspidal. 
\begin{enumerate}
    \item Assume that $B_F^K(\pi_2)$ is not cuspidal. Then $\pi_2 = I_K^F(\chi_2)$ for some character $\chi_2$ on $F$. In this case we have 
    \begin{align*}
        L(s,\pi_1 \times \pi_2 \times \pi) &= L(s,(\chi_2 \boxplus \chi_2^\theta) \times B_F^K(\pi) \otimes \chi_1) \\
        &= L(s,B_F^K(\pi) \otimes \chi_1\chi_2) \cdot L(s,B_F^K(\pi) \otimes \chi_1\chi_2^\theta)
    \end{align*}
    which has no Landau-Siegel zero by Lemma \ref{standard L-function Landau-Siegel zero} (3).
    \item Assume that $B_F^K(\pi_2)$ is cuspidal, but is not isomorphic to $B_F^K(\pi_2) \otimes \chi_1^\theta\chi_1^{-1}$. Then by Lemma \ref{Rankin-Selberg product functorality} (1) we know that $\pi_1 \boxtimes \pi_2$ is cuspidal. We write 
    $$L(s,\pi_1 \times \pi_2 \times \pi) = L(s,(\pi_1 \boxtimes \pi_2) \times \pi).$$
    Note that $(\pi_1 \boxtimes \pi_2) \otimes \eta = (\pi_1 \otimes \eta) \boxtimes \pi_2 = \pi_1 \boxtimes \pi_2$, and $\pi$ is not isomorphic to $\pi \otimes \eta$. By Lemma \ref{new Rankin-Selberg Landau-Siegel zero} (2), $L(s,(\pi_1 \boxtimes \pi_2) \times \pi)$ has no Landau-Siegel zero.
    \item Assume that $B_F^K(\pi_2)$ is cuspidal, and is isomorphic to $B_F^K(\pi_2) \otimes \chi_1^\theta\chi_1^{-1}$. Let $L/K$ be the quadratic extension associated with $\chi_1^\theta\chi_1^{-1}$. Then $B_F^K(\pi_2) = I_L^K(\xi_2)$ for some character $\xi_2$ on $L$. In this case we have 
    \begin{align*}
        L(s,\pi_1 \times \pi_2 \times \pi) &= L(s,I_L^K(\xi_2) \times B_F^K(\pi) \otimes \chi_1) \\
        &= L(s,B_K^L(B_F^K(\pi) \otimes \chi_1) \otimes \xi_2)
    \end{align*}
    which has no Landau-Siegel zero by Lemma \ref{standard L-function Landau-Siegel zero} (3).
\end{enumerate}

\section*{Acknowledgment}
We would like to thank Prof. Wenzhi Luo for his suggestion on the topic and his valuable comments. We would also like to thank Prof. Jesse Thorner for helpful conversations.

\bibliographystyle{alpha}	
\bibliography{ref.bib}

\begin{thebibliography}{Tho25b}

\bibitem[AC89]{ArthurClozel1989}
James Arthur and Laurent Clozel.
\newblock {\em Simple algebras, base change, and the advanced theory of the trace formula}, volume 120 of {\em Annals of Mathematics Studies}.
\newblock Princeton University Press, Princeton, NJ, 1989.

\bibitem[Ban97]{Banks1997}
William~D. Banks.
\newblock Twisted symmetric-square {$L$}-functions and the nonexistence of {S}iegel zeros on {${\rm GL}(3)$}.
\newblock {\em Duke Math. J.}, 87(2):343--353, 1997.

\bibitem[Dav00]{Davenportbook}
Harold Davenport.
\newblock {\em Multiplicative number theory}, volume~74 of {\em Graduate Texts in Mathematics}.
\newblock Springer-Verlag, New York, third edition, 2000.
\newblock Revised and with a preface by Hugh L. Montgomery.

\bibitem[Gar87]{Garrett1987}
Paul~B. Garrett.
\newblock Decomposition of {E}isenstein series: {R}ankin triple products.
\newblock {\em Ann. of Math. (2)}, 125(2):209--235, 1987.

\bibitem[GJ78]{GelbartJacquet1978}
Stephen Gelbart and Herv\'e Jacquet.
\newblock A relation between automorphic representations of {${\rm GL}(2)$}\ and {${\rm GL}(3)$}.
\newblock {\em Ann. Sci. \'Ecole Norm. Sup. (4)}, 11(4):471--542, 1978.

\bibitem[HB83]{Heath-Brown1983}
D.~R. Heath-Brown.
\newblock Prime twins and {S}iegel zeros.
\newblock {\em Proc. London Math. Soc. (3)}, 47(2):193--224, 1983.

\bibitem[HL94]{HoffsteinLockhart1994}
Jeffrey Hoffstein and Paul Lockhart.
\newblock Coefficients of {M}aass forms and the {S}iegel zero.
\newblock {\em Ann. of Math. (2)}, 140(1):161--181, 1994.
\newblock With an appendix by Dorian Goldfeld, Hoffstein and Daniel Lieman.

\bibitem[HR95]{HoffsteinRamakrishnan1995}
Jeffrey Hoffstein and Dinakar Ramakrishnan.
\newblock Siegel zeros and cusp forms.
\newblock {\em Internat. Math. Res. Notices}, pages 279--308, 1995.

\bibitem[IK04]{IwaniecKowalskiANT}
Henryk Iwaniec and Emmanuel Kowalski.
\newblock {\em Analytic number theory}, volume~53 of {\em American Mathematical Society Colloquium Publications}.
\newblock American Mathematical Society, Providence, RI, 2004.

\bibitem[IS00]{IwaniecSarnak2000}
Henryk Iwaniec and Peter Sarnak.
\newblock The non-vanishing of central values of automorphic {$L$}-functions and {L}andau-{S}iegel zeros.
\newblock {\em Israel J. Math.}, 120:155--177, 2000.

\bibitem[JPSS83]{JacquetP-SShalika1983}
H.~Jacquet, I.~I. Piatetskii-Shapiro, and J.~A. Shalika.
\newblock Rankin-{S}elberg convolutions.
\newblock {\em Amer. J. Math.}, 105(2):367--464, 1983.

\bibitem[Kim03]{Kim2003fourth}
Henry~H. Kim.
\newblock Functoriality for the exterior square of {${\rm GL}_4$} and the symmetric fourth of {${\rm GL}_2$}.
\newblock {\em J. Amer. Math. Soc.}, 16(1):139--183, 2003.
\newblock With appendix 1 by Dinakar Ramakrishnan and appendix 2 by Kim and Peter Sarnak.

\bibitem[KS02a]{KimShahidi2002FourthPowerCuspidality}
Henry~H. Kim and Freydoon Shahidi.
\newblock Cuspidality of symmetric powers with applications.
\newblock {\em Duke Math. J.}, 112(1):177--197, 2002.

\bibitem[KS02b]{KimShahidi2002Cube}
Henry~H. Kim and Freydoon Shahidi.
\newblock Functorial products for {${\rm GL}_2\times{\rm GL}_3$} and the symmetric cube for {${\rm GL}_2$}.
\newblock {\em Ann. of Math. (2)}, 155(3):837--893, 2002.
\newblock With an appendix by Colin J. Bushnell and Guy Henniart.

\bibitem[LRS99]{LuoRudnickSarnak1999}
Wenzhi Luo, Ze\'ev Rudnick, and Peter Sarnak.
\newblock On the generalized {R}amanujan conjecture for {${\rm GL}(n)$}.
\newblock In {\em Automorphic forms, automorphic representations, and arithmetic ({F}ort {W}orth, {TX}, 1996)}, volume 66, Part 2 of {\em Proc. Sympos. Pure Math.}, pages 301--310. Amer. Math. Soc., Providence, RI, 1999.

\bibitem[Luo23]{Luo2023}
Wenzhi Luo.
\newblock Non-existence of {S}iegel zeros for cuspidal functorial products on {$GL(2) \times GL(3)$}.
\newblock {\em Proc. Amer. Math. Soc.}, 151(5):1915--1919, 2023.

\bibitem[NT25]{NewtonThorne2025}
James Newton and Jack~A. Thorne.
\newblock Symmetric power functoriality for hilbert modular forms, 2025.

\bibitem[PSR87]{P-SRallis1987}
I.~Piatetski-Shapiro and Stephen Rallis.
\newblock Rankin triple {$L$} functions.
\newblock {\em Compositio Math.}, 64(1):31--115, 1987.

\bibitem[Ram00]{Ramakrishnan2000}
Dinakar Ramakrishnan.
\newblock Modularity of the {R}ankin-{S}elberg {$L$}-series, and multiplicity one for {${\rm SL}(2)$}.
\newblock {\em Ann. of Math. (2)}, 152(1):45--111, 2000.

\bibitem[Ram14]{Ramakrishnan2014}
Dinakar Ramakrishnan.
\newblock An exercise concerning the selfdual cusp forms on {${\rm GL}(3)$}.
\newblock {\em Indian J. Pure Appl. Math.}, 45(5):777--785, 2014.

\bibitem[RW03]{RamakrishnanWang2003}
Dinakar Ramakrishnan and Song Wang.
\newblock On the exceptional zeros of {R}ankin-{S}elberg {$L$}-functions.
\newblock {\em Compositio Math.}, 135(2):211--244, 2003.

\bibitem[RW04]{RamakrishnanWang2004}
Dinakar Ramakrishnan and Song Wang.
\newblock A cuspidality criterion for the functorial product on {$\rm GL(2)\times GL(3)$} with a cohomological application.
\newblock {\em Int. Math. Res. Not.}, 163(27):1355--1394, 2004.

\bibitem[Tak14]{Takeda2014}
Shuichiro Takeda.
\newblock The twisted symmetric square {$L$}-function of {${\rm GL}(r)$}.
\newblock {\em Duke Math. J.}, 163(1):175--266, 2014.

\bibitem[Tho21]{Thorner2021}
Jesse Thorner.
\newblock Effective forms of the {S}ato-{T}ate conjecture.
\newblock {\em Res. Math. Sci.}, 8(1):Paper No. 4, 21, 2021.

\bibitem[Tho25a]{Thorner2025symmetricsquare}
Jesse Thorner.
\newblock Exceptional zeros of $\mathrm{GL}_3\times\mathrm{GL}_3$ rankin-selberg $l$-functions, 2025.

\bibitem[Tho25b]{Thorner2025}
Jesse Thorner.
\newblock Exceptional zeros of rankin-selberg $l$-functions and joint sato-tate distributions, 2025.

\end{thebibliography}

\vspace{3mm}

\noindent Department of Mathematics, The Ohio State University, Columbus, OH 43210, USA.

\noindent E-mail address: {\tt zhao.3326@osu.edu}

\end{document}